\documentclass[11 pt]{amsart}

\usepackage{amssymb, amsthm}
\usepackage{mathtools}
\usepackage{mathrsfs}

\usepackage{csquotes}
\usepackage{enumitem}

\usepackage[hidelinks]{hyperref}

\numberwithin{equation}{section}
\newtheorem{proposition}{Proposition}[section]
\newtheorem{theorem}[proposition]{Theorem}
\newtheorem{lemma}[proposition]{Lemma}

\theoremstyle{definition}
\newtheorem{definition}{Definition}[section]

\DeclareMathOperator{\R}{\mathbb{R}}
\DeclareMathOperator{\Z}{\mathbb{Z}}

\DeclareMathOperator{\Lip}{Lip}
\DeclareMathOperator{\diam}{diam}

\DeclareMathOperator{\KR}{KR}

\newcommand{\norm}[1]{\lVert#1\rVert}
\newcommand{\abs}[1]{\lvert#1\rvert}
\newcommand{\CAT}{\operatorname{CAT}}

\renewcommand{\epsilon}{\varepsilon}
\DeclareMathOperator{\N}{\mathbb{N}}

\DeclareMathOperator{\Nagata}{Nagata}
\DeclareMathOperator{\LC}{LC}
\DeclareMathOperator{\hd}{hd^+}
\DeclareMathOperator{\Whitney}{Whitney}
\DeclareMathOperator{\NPC}{gNPC}
\DeclareMathOperator{\dimN}{\dim_N}
\DeclareMathOperator{\Vol}{Vol}

\usepackage{etoolbox}

\makeatletter
\patchcmd{\@setaddresses}{\indent}{\noindent}{}{}
\patchcmd{\@setaddresses}{\indent}{\noindent}{}{}
\patchcmd{\@setaddresses}{\indent}{\noindent}{}{}
\patchcmd{\@setaddresses}{\indent}{\noindent}{}{}
\makeatother

\begin{document}

\title[Lipschitz extension theorems with explicit constants]{Lipschitz extension theorems with \\ explicit constants}
\author{Giuliano Basso}
\address{Max Planck Institute for Mathematics, Vivatsgasse 7, 53111 Bonn, Germany}
\email{basso@mpim-bonn.mpg.de}



\begin{abstract}
In this mostly expository article, we give streamlined proofs of several well-known Lipschitz extension theorems. We pay special attention to obtaining statements with explicit expressions for the extension constants. One of our main results is an explicit version of a very general Lipschitz extension theorem of Lang and Schlichenmaier. A special case of the theorem reads as follows: If \(X\) is any metric space and \(A\subset X\) satisfies the condition \(\text{Nagata}(n, c)\), then any \(1\)-Lipschitz map \(f\colon A \to Y\) to a Banach space \(Y\) admits a Lipschitz extension \(F\colon X \to Y\) whose Lipschitz constant is at most \(1000\cdot (c+1)\cdot \log_2(n+2)\). By specifying to doubling metric spaces, this recovers an extension result of Lee and Naor. We also revisit another theorem of Lee and Naor by showing that if \(A\subset X\) consists of \(n\) points, then Lipschitz extensions as above exist  with a Lipschitz constant of at most \(600 \cdot \log n \cdot (\log \log n)^{-1}\).
\end{abstract}

\maketitle

\vspace{-1em}

\section{Introduction}

In this article we revisit Lipschitz extension theorems of Lee and Naor \cite{lee--naor-2005}, and Lang and Schlichenmaier \cite{lang-schlichenmaier}.  These results should be viewed as metric analogues of Whitney's classical extension theorem \cite{whitney--1934}. They provide conditions under which a partially defined Lipschitz map on a metric space can be extended to the entire space. The nature of this
article is mostly expository. We give somewhat streamlined proofs of selected results of \cite{lee--naor-2005, lang-schlichenmaier}. However, we closely follow the original proof strategies. The \textit{raison d'être} for this article is that we made an effort to only work with explicit constants in our arguments and to carefully keep track of those constants. Thanks to this, we are able to state our results without using the usual asymptotic notation \(O(\cdot)\). We believe that finding the best possible constants in the theorems we discuss here is an interesting task, of which this article is just a first step forward. 

\subsection{Motivation and background}
In the following, we denote by \(X\) and \(Y\) metric spaces, where we call \(X\) the domain space and \(Y\) the target space.  The following general question about the extendability of Lipschitz maps is a natural object of study in metric geometry:
\vspace{0.1em}
\begin{quote}
What are conditions on the domain and target space such that any \(1\)-Lipschitz map \(f\colon A \to Y\), where \(A\subset X\) is an arbitrary subset, can be extended to an \(L\)-Lipschitz map \(F\colon X \to Y\)?
\end{quote} 
\vspace{0.1em}
Classically, researchers were mainly looking for extensions with \(L=1\). Such extensions exist for example if the target space is \(\R\) and the domain space is an arbitrary metric space. Indeed, a short computation reveals that \(F\colon X \to \R\) defined by
\[
F(x)=\inf_{a\in A}\big[ f(a)+d(a, x) ]
\]
is a \(1\)-Lipschitz extension of \(f\colon A \to \R\). This is generally considered as the first Lipschitz extension result and was proved by McShane \cite{McShane--1934} and also independently by Whitney \cite[p. 63]{whitney--1934}. Another classical result with \(L=1\) is Kirszbraun's theorem \cite{kirszbraun--1934, valentine--1945}, where both the target and domain spaces are Hilbert spaces. Also in this case the extension can be given by means of an explicit formula; see \cite{azagra--2021} and the references therein. 

A generalization of Kirszbraun's theorem in the setting of Alexandrov geometry was obtained by Lang and Schroeder \cite{alexander--2011, lang--schroeder}. They impose a lower curvature bound on the domain space and a matching upper curvature bound on the target space. In a similar spirit, by introducing the notions of Markov type and Markov cotype of Banach spaces, another deep generalization of Kirszbraun's theorem with \(L>1\) was proved by Ball \cite{ball--1992}. 

Other important results where \(L>1\) are often obtained in the following setting. Both the domain and target space are arbitrary Banach spaces, but the subset \(A\subset X\) is assumed to consist of at most \(n\) points. It is then proved that \(L\leq \alpha(n)\) for some function \(\alpha\colon\mathbb{N} \to \R\). The first Lipschitz extension result of this flavour was obtained by Marcus and Pisier in \cite{pisier--1984}. They showed that for Hilbert space targets one can take 
\[
\alpha(n)=C(p)\cdot (\log  n)^{1/p-1/2}
\]
if \(X=L_p\) for some \(1 <p <2\). Johnson and Lindenstrauss \cite{hilbert--1984} proved that for arbitrary domain spaces and Hilbert space targets, one can take \(\alpha(n)=2\cdot \sqrt{\log n}\). We remark that to prove their extension result, Johnson and Lindenstrauss relied on a certain geometric lemma, which is now known under the name Johnson--Lindenstrauss lemma and has found many applications in various areas of mathematics and computer science.

The first result for arbitrary Banach target spaces was obtained by Johnson, Lindenstrauss, and Schechtman \cite{lindenstrauss--1986}. They showed that one can take \(\alpha(n)=C\cdot \log n\). This was improved by Lee and Naor \cite{lee--naor-2005} to 
\[
\alpha(n)=C\cdot \frac{\log n}{\log\log n}.
\]
A lower bound on \(\alpha(n)\) was obtained by Naor and Rabani in \cite{rabani--2017}. Another (very general) lower bound on Lipschitz extension moduli was recently proved by Naor \cite{naor2021extension}. Roughly speaking, he showed that in any finite-dimensional Banach space there are subsets with bad Lipschitz extension properties. More precisely, there is a universal constant \(c>0\) such that the following holds. In every \(n\)-dimensional Banach space \(X\) there exist a subset \(A\subset X\) and a \(1\)-Lipschitz map \(f\colon A \to Y\) to a Banach space \(Y\), such that every Lipschitz extension \(F\colon X \to Y\) of \(f\) has Lipschitz constant at least \(n^c\). 

Many other Lipschitz extension results can be found in the books \cite{brudnyi--first-book, brudnyi--2012, cobzas--2019, wells-1975} and for some results in the linear category see e.g. \cite[Section III.B]{Wojtaszczyk--1991} or \cite[Section 8]{Tomczak--1989}. For an introduction to extension results in topology, I strongly recommend \cite{steenrod}.

\subsection{Spaces of finite Nagata dimension}

In this article, we will mostly consider arbitrary domain spaces \(X\), but we will restrict ourselves to subsets \(A\subset X\) that are finite-dimensional in a quantifiable metric sense. A few definitions are in order. A covering of a metric space is said to have \(s\)-multiplicity at most \(n\) if every subset of the space of diameter less than \(s\) meets at most \(n\) members of the covering. We emphasize that throughout this article we will always consider coverings by arbitrary subsets. The following definition is a variant of Gromov's asymptotic dimension, taking into account both large and small scale properties of the metric space. Let \(n\) be a non-negative integer and \(c\) a positive real number.

\begin{definition}\label{def:nagata}
We say that a metric space \(X\) satisfies the condition \(\Nagata(n, c)\) if for every \(s>0\) there exists a covering of \(X\) which has \(s\)-multiplicity at most \(n+1\) and every member of the covering has diameter at most \(c\cdot s\).    
\end{definition}

This definition is due to Assouad \cite{assouad--1982}, building on earlier work of Nagata \cite{nagata--1958}. Variants of Definition~\ref{def:nagata}, where only small or large scales are considered, are called linearly controlled metric dimension and Higson property, respectively. See \cite{buyalo--2007, dranishnikov--2004} for the precise definitions and further properties. 

The smallest \(n\geq 0\) such that \(X\) satisfies \(\Nagata(n, c)\) for some \(c\) is denoted by \(\dimN(X)\) and called the Nagata dimension of \(X\). The topological dimension \(\dim(X)\) of \(X\) is at most \(\dimN(X)\) (see \cite[Theorem~2.2]{lang-schlichenmaier}) and this inequality can be strict in general, as for example \(\dim(\Z)=0\), but \(\dimN(\Z)=1\). More dramatically, there are even finitely generated groups which have infinite Nagata dimension (see \cite{dydak-lang--2014}). On the other hand, examples of metric spaces \(X\) with \(\dim(X)=\dimN(X)\) include compact connected Riemannian manifolds, Euclidean buildings of finite rank, uniformly disconnected spaces in the sense of David and Semmes, and Carnot groups equipped with the Carnot-Carathéodory distance (see \cite{david--semmes--1997, lang-schlichenmaier, le-donne--2015}). 

A Hadamard manifold is a complete simply connected Riemannian manifold of non-positive sectional curvature. If the sectional curvature of a Hadamard manifold \(M\) is negatively pinched, then \(\dimN(M)=\dim(M)\), see \cite[Theorem 3.7]{lang-schlichenmaier}. Moreover, there exits a universal constant \(c\) such that every Hadamard manifold of dimension \(n=2\), \(3\) satisfies  \(\Nagata(n, c)\), see \cite[Theorem~5.7.3]{schlichenmaier--2005} and \cite[Theorem 3]{joergensen--2022}. We remark that it is an open question whether all Hadamard manifolds have finite Nagata dimension or not. 

Another important class of examples is obtained by considering minor-closed families of graphs. Very recently, building on a breakthrough result of \cite{bonamy2023asymptotic} for the asymptotic dimension,  Distel \cite{distel2023proper} and  Liu \cite{liu2023assouad} proved the following result. Let \(\mathcal{F}\)  be a minor-closed family of graphs not containing every finite graph.  Then there exists \(c=c(\mathcal{F})\) such that every member of \(\mathcal{F}\) satisfies \(\Nagata(2, c)\). Here, graphs are metrized in the usual way by equipping them with the shortest-path distance. 

Many results from classical dimension theory also apply to the Nagata dimension. For example, the product formula \(\dimN(X\times Y)\leq \dimN(X)+\dimN(Y)\) holds and also a Hurewicz type theorem for Lipschitz maps (see \cite{brodskiy--2008}). Moreover, analogous to the topological dimension, Nagata dimension can be characterized in terms of extension properties of maps to spheres. More precisely, Brodskiy, Dydak, Higes, and Mitra \cite{brodskiy--2009} showed that if \(X\) is a metric space of finite Nagata dimension, then \(\dimN(X)\leq n\) if and only if the pair \((X, S^n)\) has the Lipschitz extension property. See e.g.~\cite[Theorem 1.9.3]{engelking--1978} for analogous conditions for \(\dim(X)\leq n\).

\subsection{Lipschitz \(n\)-connected spaces}
We now proceed by introducing the type of target spaces considered by Lang and Schlichenmaier in \cite{lang-schlichenmaier}. Recall that a topological space is said to be \(n\)-connected if for every \(0 \leq m \leq n\), any continuous map from the \(m\)-sphere to the space admits a continuous extension to the \((m+1)\)-ball. Lang and Schlichenmaier adapted this definition to the metric setting as follows.

\begin{definition}
A metric space \(Y\) is called Lipschitz \(n\)-connected with constant \(\lambda\) (or satisfies the condition \(\LC(n, \lambda)\) for short) if for every \(0 \leq m \leq n\), every \(L\)--Lipschitz map \(f\colon S^{m} \to Y\) admits a \(\lambda\cdot L\)--Lipschitz extension \(F\colon B^{m+1}\to Y\). Here, \(B^{m+1}\subset \R^{m+1}\) denotes the Euclidean unit ball equipped with the induced metric and \(S^{m}=\partial B^{m+1}\) its boundary.
\end{definition}

For example, the sphere \(S^n\) is Lipschitz \((n-1)\)-connected (see \cite[Proposition 6.2.6]{schlichenmaier--2005}), but clearly not Lipschitz \(n\)-connected by Brouwer's theorem. On the other hand, the unit sphere in an infinite-dimensional Banach space is Lipschitz \(n\)-connected for every \(n\). In fact, Benyamini and Sternfeld \cite{Benyamini--1983} proved that these spheres are Lipschitz contractible and it is easy to check that any Lipschitz contractible space satisfies \(\LC(n, \lambda)\) for every \(n\) with a uniform constant \(\lambda\). In Section~\ref{sec:lipschitz-n-connectedness} we put some effort into finding optimal values of \(\lambda\). For example, we show that Banach spaces satisfy \(\LC(n, \sqrt{3})\) for every \(n\). We will prove this statement more generally for metric spaces admitting conical bicombings; see Proposition~\ref{prop:one}. Thus, for example, it is also valid for all Hadamard manifolds. 

There are also many examples of metric spaces satisfying \(\LC(n , \lambda)\) which do not admit conical bicombings. By results of Lang--Schlichenmaier \cite[Theorem 5.1]{lang-schlichenmaier} and Wenger--Young \cite[Theorem 1.1]{wenger--young--2010}, other examples of Lipschitz \(n\)-connected spaces are \(n\)-connected closed Riemannian manifolds, and  Carnot groups that can be realized as a jet space Carnot group \(J^s(\R^{n+1})\). Also certain horospheres in symmetric spaces of noncompact type of rank \(n\geq 2\) are Lipschitz \((n-2)\)-connected; see \cite{leuzinger--2017, young--2014}.

\subsection{Lang and Schlichenmaier's theorem} In view of the many examples stated above, the following explicit version of Lang and Schlichenmaier's extension theorem \cite[Theorem 1.6]{lang-schlichenmaier} seems to cover many of the possible settings in which one might wish to construct Lipschitz extensions.

\begin{theorem}[Lang--Schlichenmaier]\label{thm:LS-msot-general-version}
Let \(X\) be a metric space and \(A\subset X\) a closed subset satisfying \(\Nagata(n, c)\). Suppose \(Y\) is a metric space that is Lipschitz \(n\)-connected with constant \(\lambda\). Then every \(1\)-Lipschitz map \(f\colon A\to Y\) admits a
\begin{equation}\label{eq:estimate-we-worked-for-really-hard}
    10^{10^{10}} \cdot \lambda^{n+1}\cdot (c+1)^{10}\cdot (n+1)^{10\cdot n}
\end{equation}
Lipschitz extension \(F\colon X \to Y\).
\end{theorem}

The reader may be surprised at the appearance of the exponential tower \(10^{10^{10}}\) in \eqref{eq:estimate-we-worked-for-really-hard}. As it turns out, this term is only an artifact of the proof. It arises because in an intermediate step to \eqref{eq:estimate-we-worked-for-really-hard} we encounter a term of the form \((C\cdot n)^n\) which we decided to estimate by a large multiple of \(n^n\), so that \eqref{eq:estimate-we-worked-for-really-hard} does not consist of too many terms. In fact \eqref{eq:estimate-we-worked-for-really-hard} can be replaced by the slightly more complicated bound
\begin{equation}\label{eq:more-complicated}
3 \cdot 10^{10}\cdot (c+1)^{10}\cdot (10^5 \cdot \lambda)^{n+1} \cdot (n+1)^{6\cdot n}.
\end{equation}
Depending on the value of \(n\) one might prefer \eqref{eq:more-complicated} over \eqref{eq:estimate-we-worked-for-really-hard}, or vice versa. We remark that in the special case when \(n=0\), a much better bound than \eqref{eq:more-complicated} was obtained in \cite{basso--sidler}. There, it is shown that if \(n=0\), then \(f\) admits an \(8 c\lambda\)-Lipschitz extension (which factors through a metric \(\R\)-tree). 

Our proof of Theorem~\ref{thm:LS-msot-general-version} closely follows Lang and Schlichenmaier's original proof. The main idea is to construct an extension \(F\colon X \to Y\) such that \(F\) restricted to \(X\setminus A\)  factors through a simplicial complex \(\Sigma\) of dimension at most \(n+1\). This idea has its roots in classical topology where it is used for example in the proof of the Kuratowski--Dugundji extension theorem \cite{ dugundji--1958, kuratowski--1935}.  The complex \(\Sigma\) will be the nerve complex of a covering of \(X\setminus A\). This covering has properties that mimic the classical Whitney cube decomposition of an open set in Euclidean space. The map \(F\) is then obtained by successive extensions of a certain map defined on the \(0\)-skeleton of \(\Sigma\). The constant \(\approx \lambda^n \cdot n^n\) appearing in \eqref{eq:estimate-we-worked-for-really-hard} is due to the fact that we need to apply \(\LC(n, \lambda)\) a total of \((n+1)\)-times to inductively extend this map to the \((n+1)\)-skeleton of \(\Sigma\), which agrees with \(\Sigma\).

\subsection{Target spaces of generalized non-positive curvature} For many target spaces of interest such as Banach spaces or Hadamard manifolds, the above inductive application of \(\LC(n, \lambda)\) can be avoided. The main reason for this is that (informally speaking) such spaces admit “barycentric coordinates". By means of these coordinates, Lipschitz maps defined on the vertices of a simplicial complex can be extended to the entire complex in a single step. For example, if \(Y\) is a Banach space and \(f\colon \Sigma^{(0)}\to Y\) is defined on the vertices \(\Sigma^{(0)}\) of a pure \(n\)-dimensional simplicial complex \(\Sigma\), then \(F\colon \Sigma \to Y\) which on each \(n\)-simplex \(\Delta \subset \Sigma\) is defined by
\[
F(\sum_{i=0}^{n} \alpha_i v_i)=\sum_{i=0}^{n} \alpha_i f(v_i)
\]
extends \(f\). Moreover, this extension has the same Lipschitz constant as \(f\) if \(\Sigma \subset \R^I\) is equipped with the metric induced by the \(\ell_1\)-norm on \(\R^{I}\). 

Using a barycenter construction going back to Es-Sahib and Heinich \cite{es-sahib-1999}, one can show that any \(Y\) satisfying the following definition admits “barycentric coordinates" and thus extensions as above exist for such target spaces.

\begin{definition}\label{def:generalized-non-positive-curvature}
We say that a metric space \(Y\) is a space of generalized non-positive curvature (a \(\NPC\) space for short) if it admits a symmetric map \(m\colon Y \times Y \to Y\) such that \(m(y,y)=y\) for all \(y\in Y\), and also
\begin{equation}\label{eq:conical}
d(m(x, y), m(x, z))\leq  \frac{1}{2} \cdot d(y,z)    
\end{equation}
for all \(x\), \(y\), \(z\in Y\).
\end{definition}

It follows directly from the definition that \(m(x, y)\) is a metric midpoint of \(x\) and \(y\). In particular, any complete \(\NPC\) space is a geodesic metric space (see \cite[p. 4]{bridson--1999} for the definition). We remark that for midpoints in Euclidean space, the inequality in \eqref{eq:conical} becomes an equality. Therefore, for complete metric spaces, \eqref{eq:conical} can be interpreted as follows.  For any three points \(x\), \(y\), \(z\) in \(Y\) there is a geodesic triangle in \(Y\) with vertices \(x\), \(y\), \(z\) such that the distances between the midpoints of the sides are no greater than the corresponding distances between the midpoints in the Euclidean comparison triangle. It may be helpful to compare this with the definition of a \(\CAT(0)\) space, where it is required that all geodesic triangles in the space are not thicker than their comparison triangles in the Euclidean plane.

Examples of spaces of generalized non-positive curvature are Banach spaces (with \(m(x, y)=\tfrac{1}{2}(x+y)\)) and \(\CAT(0)\) spaces, where \(m(x, y)\) denotes the unique midpoint of \(x\) and \(y\). Additional examples are Busemann spaces \cite{Papadopoulos--2005} and injective metric spaces \cite{lang--2013}. On the other hand, complete \(\NPC\) spaces are contractible, and thus the circle \(S^1\) is not a \(\NPC\) space. A complete characterization of \(\NPC\) spaces as certain “convex" subsets of injective metric spaces can be found in \cite{basso2020extending}. The class of \(\NPC\) spaces enjoys many good structural properties. For example, it is closed under Gromov-Hausdorff limits, \(\ell_p\)-products for \(p \in [1, \infty]\) and  also \(1\)-Lipschitz retractions.

For target spaces of generalized non-positive curvature, the upper bound on the Lipschitz constant in Theorem~\ref{thm:LS-msot-general-version} can be improved considerably.

\begin{theorem}\label{thm:Lang-Schlichenmaier-explicit}
Let \(X\) be a metric space and \(A\subset X\) a subset satisfying \(\Nagata(n, c)\). Suppose that \(Y\) is a complete metric space of generalized non-positive curvature. Then every \(1\)-Lipschitz map \(f\colon A\to Y\) admits a
\begin{equation}\label{eq:estimate-of-F}
1000\cdot(c+1)\cdot\log_2(n+2)
\end{equation}
Lipschitz extension \(F\colon X\to Y\).
\end{theorem}

A version of the theorem with \eqref{eq:estimate-of-F} replaced by \(O(c\cdot n^3)\) was proved by Naor and Silberman in \cite[Corollary 5.2]{naor--silbermann--2011}. Their result resolves a conjecture of Brudnyi and Brudnyi \cite[p. 48]{brudnyi--2012}, who asked whether, for spaces \(A\subset X\) satisfying  \(\Nagata(n, c)\) and \(\CAT(0)\) target spaces \(Y\), it is always possible to extend \(1\)-Lipschitz maps \(A\to Y\) to Lipschitz maps \(X\to Y\), whose Lipschitz constants are bounded by \(K \cdot n^a\), for some \(K\), \(a>0\) depending only on the Nagata constant \(c\). The sublinear dependence of \eqref{eq:estimate-of-F} on the parameters \(n\),~\(c\) shows that Lipschitz extensions in this setting have even better properties than what was generally expected.

The main idea behind the proofs of the Theorems~\ref{thm:LS-msot-general-version} and \ref{thm:Lang-Schlichenmaier-explicit} is the same. In particular, the restriction of \(F\) to \(X\setminus A\) thus factorizes by construction through a simplicial complex \(\Sigma\), which is the nerve complex of a Whitney-type covering of \(X\setminus A\). The appearance of \(\log_2(n+2)\) in \eqref{eq:estimate-of-F} stems from the fact that we use a method of Johnson, Lindenstrauss and Schechtman \cite{lindenstrauss--1986} to construct particularly good Lipschitz partitions of unity. The main difference between the proofs is that, as mentioned above, we use “barycentric coordinates" to extend maps from \( \Sigma^{(0)}\) to \(Y\) in one single step. More precisely, by a result of Descombes \cite{descombes--2016} we know that spaces of generalized non-positive curvature \(Y\) admit a barycenter map \(\beta\colon \mathcal{P}_1(Y)\to Y\) in the sense of Sturm \cite[Remark 6.4]{sturm--2003}. Here, \(\mathcal{P}_1(Y)\) denotes the space of Radon probability measures on \(Y\) with a bounded first moment. Using such a barycenter map \(\beta\), one can extend any \(f\colon \Sigma^{(0)}\to Y\) to a map \(F\colon \Sigma \to Y\) by setting
\[
F(\sum_{i=0}^{n} \alpha_i v_i)=\beta\big( \sum_{i=0}^{n} \alpha_i \, \delta_{f(v_i)}\big)
\]
on each \(n\)-simplex \(\Delta\subset \Sigma\). Since \(\beta\) is \(1\)-Lipschitz with respect to the \(1\)-Wasserstein distance, this extension has the desired properties. The necessary background on Wasserstein distances is recalled in Section~\ref{sec:spaces-of-generalized-non-po-curvature}. 

For a given domain space \(X\), for which it is known from general considerations that it satisfies \(\Nagata(n, c)\) for some \(c\), it is often very difficult to obtain a good estimate on \(c\). For example, already the asymptotic behaviour of the smallest \(c\) such that \(\R^n\) satisfies \(\Nagata(n, c)\) seems to be unknown. Other instances where \(n\) is known but the estimation of \(c\) is difficult are closed Riemannian \(n\)-manifolds. For such spaces, the following Lipschitz extension result sometimes yields better bounds than Theorem~\ref{thm:Lang-Schlichenmaier-explicit}.

\begin{theorem}\label{thm:bilipschitz-triangulation}
Let \(X\) be a compact connected Riemannian \(n\)-manifold for some \(n\geq 2\). Suppose that \(X\) admits a triangulation \(\Sigma \to X\) which is \(D\)-bi-Lipschitz when restricted to any \(n\)-simplex and the number of \(n\)-simplices of \(\Sigma\) is bounded by \(N\). Then every \(1\)-Lipschitz map \(f\colon A \to Y\) defined on \(A\subset X\) mapping to a complete \(\CAT(0)\) space \(Y\) admits a 
\[
D\cdot N^{10 \log(n)}
\]
Lipschitz extension \(F\colon X \to Y\).
\end{theorem}

The proof of the theorem is quite short and relies on straightforward properties of simplicial complexes and the generalized Kirszbraun theorem of Lang and Schroeder \cite{lang--schroeder}.

\subsection{Extension results of Lee and Naor} We now turn our attention to some of the results of Lee and Naor from \cite{lee--naor-2005}. A metric space \(X\) is called \(M\)-doubling if every ball in \(X\) can be covered by at most \(M\) balls of half the radius. It follows directly from \cite[Lemma~2.3]{lang-schlichenmaier} that any \(M\)-doubling metric spaces satisfies \(\Nagata(M^3, 2)\). Thus, all the Lipschitz extension theorems above apply specifically to doubling metric spaces. In particular, the following explicit version of \cite[Theorem~1.6]{lee--naor-2005} is a direct consequence of Theorem~\ref{thm:Lang-Schlichenmaier-explicit}.

\begin{theorem}[Lee--Naor -- doubling spaces]\label{thm:lee-and-naor-doubling} Let \(X\) be a metric space and \(A\subset X\) a subset that is \(M\)-doubling for some \(M\geq 2\). Suppose that \(Y\) is a Banach space. Then every \(1\)-Lipschitz map \(f\colon A\to Y\) admits a
\begin{equation*}
10^5\cdot \log(M) 
\end{equation*}
Lipschitz extension \(F\colon X\to Y\).
\end{theorem}

A streamlined proof of this theorem has already been given by Brudnyi and Brudnyi in \cite{brudnyi--2006}. Their proof is constructive and uses a Lipschitz extension operator based on the integral average over balls. However, the obtained universal constant is not specified, and it probably requires some work to make it explicit. A generalization of Theorem~\ref{thm:lee-and-naor-doubling} to spaces of pointwise homogeneous type, which include, for example, the standard hyperbolic spaces \(\mathbb{H}^n\), can be found in \cite{brudnyi--2007} and \cite[Section 7]{brudnyi--2012}. It follows from Theorem~\ref{thm:lee-and-naor-doubling} that if \(A\) is a finite metric space which consists of at most \(n\) points, then any \(1\)-Lipschitz map \(f\colon A \to Y\) admits a 
\[
10^5\cdot \log(n)
\]
Lipschitz extension \(F\colon X \to Y\). A deep theorem of Lee and Naor \cite{lee--naor-2005} improved the asymptotics of this bound. We have the following explicit version of \cite[Theorem 1.10]{lee--naor-2005}.

\begin{theorem}[Lee--Naor -- finite spaces]\label{thm:lee-and-naor}
Let \(X\) be a metric space and \(A\subset X\) a subset consisting of at most \(n\) points. Suppose that \(Y\) is a Banach space. Then every \(1\)-Lipschitz map \(f\colon A \to Y\) admits a
\begin{equation}\label{eq:lee-and-naor-best-estimate}
600 \cdot \frac{\log(n)}{\log(\log(n))}
\end{equation}
Lipschitz extension \(F\colon X \to Y\).
\end{theorem}

Lee and Naor's proof uses an elegant multiscale extension argument and relies on stochastic decomposition techniques from theoretical computer science. A rough sketch of their proof strategy can be found in Section~\ref{sec:multiscale}, where we essentially follow this strategy and give a self-contained proof of the theorem. The proof crucially relies on certain padded decompositions of finite metric spaces. We construct these decompositions in Lemma~\ref{lem:itertive-ball-partitioning} using the methods of \cite{calinescu--2004, Fakcharoenphol--2003, Fakcharoenphol--2004}.  Contrary to what is common in the literature, our formulation of the lemma does not use any stochastic language. Moreover, our proof can be viewed as a de-randomized version of the usual proof. However, it is mathematically equivalent to the usual stochastic proof and therefore the whole reformulation is essentially just a matter of linguistics.

Theorem~\ref{thm:lee-and-naor} admits a concise reformulation using the absolute extendability constant of a metric space.

\begin{definition}
A metric space \(X\) is called absolutely Lipschitz extendable if there exists a constant \(L\geq 0\) such that the following holds. For every metric space \(X^e\) containing \(X\) and every Banach space \(Y\), any \(1\)-Lipschitz map \(f\colon X\to Y\) can be extended to an \(L\)-Lipschitz map \(F\colon X^e\to Y\). The smallest such \(L\geq 0\) is denoted by \(\mbox{\ae}(X)\). 
\end{definition}

Hence, the theorem above states that if \(X\) consists of at most \(n\) points, then
\begin{equation}\label{eq:absolute-extendability}
\mbox{\ae}(X) \leq 600 \cdot \frac{\log(n)}{\log(\log(n))}.
\end{equation}
For small \(n\), this estimate can be further improved. It follows directly from \cite[Theorem 1.4]{basso--2022} and the Kadets--Snobar theorem \cite{kadets--1972} that
\[
\mbox{\ae}(X) \leq \sqrt{n}
\]
for any \(n\)-point metric space \(X\). This estimate is stronger than \eqref{eq:absolute-extendability} for any \(n\leq 10^7\). Let us remark here that little is known about the precise values of \(\mbox{\ae}(n)=\sup\{ \mbox{\ae}(X) : \abs{X}=n \}\) for small \(n\). It is conceivable that good bounds of this constants for small \(n\) could shed some light on the general behaviour of the function \(\mbox{\ae}(n)\).

\subsection{Structure of the article} This paper is organized as follows. In Section~\ref{sec:preliminaries} we prove two lemmas relating pointwise Lipschitz constants to (global) Lipschitz constants. We also recall the necessary background from optimal transport theory regarding our use of “barycentric coordinates". Section~\ref{sec:whitney} is devoted to the proof of  Theorem~\ref{thm:Lang-Schlichenmaier-explicit}. The proof of this theorem can be seen as a blueprint for the more elaborate proof of  Theorem~\ref{thm:LS-msot-general-version} in Section~\ref{sec:explicit-lang-schlichenmaier}. In the next section we prove Theorem~\ref{thm:lee-and-naor}, our explicit version of the extension theorem of Lee and Naor. The main purpose of Section~\ref{sec:geometry-of-simplicial-complexes} is to introduce terminology regarding simplicial complexes. Using some of the tools developed there, we also give the straightforward proof of Theorem~\ref{thm:bilipschitz-triangulation}. In Section~\ref{sec:general-whitney} we establish a general extension theorem relating Lipschitz extension constants to parameters of certain coverings, which we call Whitney coverings. Section~\ref{sec:lipschitz-n-connectedness} deals with Lipschitz \(n\)-connected spaces. We discuss alternative definitions there and show in Proposition~\ref{prop:one} that \(\NPC\)-spaces are Lipschitz  \(n\)-connected for every \(n\). Finally, in Section~\ref{sec:explicit-lang-schlichenmaier}, we prove Theorem~\ref{thm:LS-msot-general-version} by the use of Theorem~\ref{thm:basso-theorem}.

\section{Preliminaries}\label{sec:preliminaries}

\subsection{Basic metric notions}
Let \(X=(X, d)\) be a metric space. For any \(x\in X\) and \(r>0\), we use \(B(x, r)\) to denote the closed ball in \(X\) with center \(x\) and radius \(r\). Given two subsets \(A\), \(B\subset X\), we write
\[
d(A, B)=\inf\{ d(a, b) : a\in A \text{ and } b\in B\}
\]
to denote the infimal distance between \(A\) and \(B\). To ease notation, we write \(d(x, B)\) for \(d(\{x\}, B)\). We use
\[
N_r(B)=\{ x\in X : d(x, B) <r \}
\]
to denote the open \(r\)-neighborhood of \(B\) in \(X\). A covering of \(X\) is by definition a collection \((B_i)_{i\in I}\) of subsets of \(X\) such that \(\bigcup_{i\in I} B_i=X\). For all coverings constructed in this article, we tacitly assume that they consist only of non-empty subsets. This does not cause any problems because whenever \(X\) is non-empty and the empty set is removed from a covering of \(X\), the resulting collection is still a covering.  

\subsection{Pointwise Lipschitz constants}
For any map \(f\colon X\to Y\) between metric spaces, we denote by
\[
\Lip f(x)=\limsup_{x'\to x} \frac{d(f(x), f(x'))}{d(x,x')}
\]
the (upper) pointwise Lipschitz constant of \(f\) at \(x\). The following lemma shows that in a quasiconvex domain space,
one can prove Lipschitz continuity by establishing an upper bound on the pointwise Lipschitz constants. Recall that \(X\) is \(c\)-quasiconvex if any two points in \(X\) can be joined by a continuous path of length at most \(c\) times their distance.

\begin{lemma}\label{lem:folklore}
Let \(f\colon X \to Y\) be a map between metric spaces. If \(X\) is \(c\)-quasiconvex and \(\Lip f(x) \leq L\) for every \(x\in X\), then \(f\) is \(cL\)-Lipschitz. 
\end{lemma}
\begin{proof}
This is a folklore result, see e.g.~\cite[Lemma~2.4]{basso--wenger--young} for a proof.
\end{proof}

The object of study of this article are extensions \(F\colon X \to Y\) of a given Lipschitz map \(f\colon A \to Y\). Since we consider arbitrary domain spaces, by Kuratowski's embedding there is no loss in generality to assume that \(X\) is a Banach space (and thus in particular \(1\)-quasiconvex). Hence, by the lemma above, to check whether \(F\) is Lipschitz it suffices to bound the pointwise Lipschitz constant of \(F\). Our next lemma simplifies this even further and shows that for continuous \(F\), it is sufficient to only bound the pointwise Lipschitz constants of points in \(X\setminus A\). This lemma is very useful and is applied in all the proofs of our main theorems.

\begin{lemma}\label{lem:great-simplification}
Let \(X\) be a Banach space and \(f\colon A \to Y\) a \(1\)-Lipschitz map defined on a closed subset \(A\subset X\) to a metric space \(Y\). Suppose \(F\colon X \to Y\) is an extension of \(f\) that is continuous at every point of \(A\) and there exists \(L\geq 1\) such that for every \(x\in X \setminus A\),
\begin{equation}\label{eq:bound-on-x}
\Lip F(x)\leq L.
\end{equation}
Then \(F\colon X \to Y\) is \(L\)-Lipschitz.
\end{lemma}

\begin{proof}
In view of  Lemma~\ref{lem:folklore} it suffices to show that
\begin{equation}\label{eq:bound-on-a}
\Lip F(a) \leq L
\end{equation}
for all \(a\in A\). Indeed, if \eqref{eq:bound-on-a} holds then the assumptions of Lemma~\ref{lem:folklore} are satisfied as Banach spaces are \(1\)-quasiconvex and we have \(\Lip F(x) \leq L\) for all \(x\in X\) because of \eqref{eq:bound-on-x} and \eqref{eq:bound-on-a}. Let \(a\in A\) and \(x\in X\setminus A\), and consider the geodesic segment 
\[
x_t= (1-t) x+t a
\]
for \(t\in [0,1]\). Since \(A\) is closed and \(x\notin A\), there exists a smallest real number \(s\in (0,1]\) such that \(x_s\in A\). By \eqref{eq:bound-on-x} we have \(\Lip F(x_t)\leq L\) for every \(t\in [0, s)\). Thus, it follows from Lemma~\ref{lem:folklore}, that \(d(F(x), F(x_t))\leq L\cdot d(x, x_t)\) for any such \(t\). Hence, 
\[
d(F(x), F(a))\leq L\cdot d(x, x_t)+d(F(x_t), F(x_s))+d(x_s, a)
\]
and so 
\[
d(F(x), F(a)) \leq L\cdot d(x, a).
\]
since \(F\) is continuous at \(x_s\) and \(t<s\) can be chosen arbitrarily close to \(s\). This yields \eqref{eq:bound-on-a} as \(L\geq 1\).
\end{proof}

\subsection{Good coverings}
Given a covering of a metric space one might ask how often a given point is contained well within the interior of the members of the covering. For a covering \((B_i)_{i\in I}\), a fixed point \(x\in X\), and \(D>0\), the ratio
\[
\delta_x(D)=\frac{\#\{ i\in I : B(x, D)\subset B_i\}}{\#\{ i\in I : x\in B_i\}}
\]
tells us the percentage of the sets from our collection that contain \(x\) in a way that it is at least distance \(D\) away from their boundaries. For our Lipschitz extension theorems, it turns out to be helpful to work with coverings where this ratio is as large as possible. We informally refer to such coverings as good coverings. For example, spaces that satisfy \(\Nagata(n, c)\) admit good coverings. More precisely, such spaces \(X\) admit for any \(D>0\) a \(2c D\)-bounded covering such that \(\delta_x(D)\geq \frac{1}{n+1}\) for every \(x\in X\). This can easily be seen by looking at the open \((D/2)\)-neighborhoods of the members of the Nagata cover at scale \(D\). This simple observation will be helpful in the proof of Theorem~\ref{thm:Lang-Schlichenmaier-explicit} in Section~\ref{sec:whitney}.

The proof of the Lee--Naor extension theorem in Section~\ref{sec:multiscale} requires the following more refined estimate on \(\delta_x(D)\) for finite metric spaces.

\begin{lemma}\label{lem:itertive-ball-partitioning}
Let \(X\) be a finite metric space and \(D>0\). Then \(X\) admits a finite, \(D\)-bounded covering \((B_i)_{i\in I}\) such that
\begin{equation}\label{eq:padded-1}
 \frac{\#\{ i\in I : B(x, \frac{D}{4})\subset B_i\}}{\#\{ i\in I : x\in B_i\}} \geq \frac{\# B(x, \frac{D}{4})}{\# B(x, D)}
\end{equation}
for all \(x\in X\).
\end{lemma}

\begin{proof}
Our proof relies on an iterative ball partitioning method, which is highly influential and extensively used by many authors in the field of theoretical computer science. Notable references include \cite{bartal1996probabilistic, calinescu--2004, Fakcharoenphol--2003, Fakcharoenphol--2004, linial1991decomposing}. 

Fix an enumeration \(X=\{x_1, \dots, x_n\}\) and let \(S_n\) denote the symmetric group of order \(n\). For each \(\pi\in S_n\) we define \(B_\pi^1, \dots, B_\pi^n\subset X\) by the following recursive rule: \(B_\pi^1=B(x_{\pi(1)}, D)\) and \(B_\pi^k=B(x_{\pi(k)}, D)\setminus B_\pi^1\cup \dots \cup B_\pi^{k-1}\) for all \(k>1\). Clearly, each \(B_\pi^k\) is \(2D\)-bounded and \(\{ B_\pi^1, \dots, B_\pi^n\}\) is a partition of \(X\). Thus, \(\mathcal{B}=(B_\pi^k)\) is a \(2D\)-bounded covering of \(X\) such that each \(x\in X\) is contained in exactly \(n!\) members of \(\mathcal{B}\). 

In what follows, we show \eqref{eq:padded-1}. Fix \(x\in X\). Without loss of generality we may suppose that \(x=x_1\) and \(d(x, x_\ell) \leq d(x, x_k)\) whenever \(\ell <k\).  Clearly,
\begin{equation}\label{eq:first-ineq}
\delta_x\bigl(\tfrac{D}{2} \bigr)\geq \frac{1}{n!} \sum_{k=1}^{\# B(x, \frac{D}{2})} \# A_k    
\end{equation}
where \(A_k= \bigl\{ \pi \in S_n : B(x, \frac{D}{2}) \subset B_\pi^k \bigr\}\). For any \(k\) with \(d(x, x_k) \leq \frac{D}{2}\) one has that \(\pi \in A_k\) whenever \(\pi(k) \leq \pi (\ell)\) for all \(\ell=1, \dots, \# B(x, 2D)\). In particular, letting \(m= \# B(x, 2D)\), we find that
\begin{equation}\label{eq:second-ineq}
\# A_k \geq \binom{n}{m} (m-1)!\, (n-m)!=\frac{n!}{m}
\end{equation}
for all \(k=1, \dots, \# B(x, \frac{D}{2})\). Therefore, by combining \eqref{eq:first-ineq} with \eqref{eq:second-ineq} we arrive at \eqref{eq:padded-1}, as desired. 
\end{proof}

\subsection{Spaces of generalized non-positive curvature}\label{sec:spaces-of-generalized-non-po-curvature}  Let \(X\) be a metric space and denote by \(\mathcal{B}(X)\) its Borel \(\sigma\)-algebra. A signed measure \(\lambda \colon \mathcal{B}(X) \to \R\) is said to be a real Radon measure if its total variation \(\mu=\abs{\lambda}\) is an inner regular measure, that is, \(\mu(B)\) is equal to the supremum \(\mu(K)\) over all compact subset \(K\subset B\). We denote by \(\mathcal{M}(X)\) the vector space of all real Radon measures on \(X\). Let \(\mathcal{M}_s(X)\) be the linear subspace of \(\mathcal{M}(X)\) generated by all measures of the form \(\delta_x-\delta_y\) for \(x\), \(y\in X\). The following expression 
\[
\norm{\lambda}_{\KR}=\sup\Big\{ \int_X f(x)\, \lambda(dx) : f\in \Lip_1(X, \R) \Big\}
\]
clearly defines a norm on \(\mathcal{M}_s(X)\). This norm is called the Kantorovich-Rubinstein norm and the completion of \(\mathcal{M}_s(X)\) with respect to \(\norm{\, \cdot \,}_{\KR}\) is called the Lipschitz free space over \(X\) and is denoted by \(\mathcal{F}(X)\). If \((X, p)\) is a pointed metric space, then \(\delta\colon X \to \mathcal{F}(X)\) defined by \(\delta(x)=\delta_x-\delta_p\) is an isometric embedding. For Banach spaces, this map is not linear, but it admits a bounded \emph{linear} inverse \(\beta\colon \mathcal{F}(X) \to X\) which satisfies \(\beta(\lambda)=\int x \,d\lambda\) for all \(\lambda\in \mathcal{M}_s(X)\). This barycenter map \(\beta\) is a key notion in the theory of Lipschitz free spaces in Banach space theory \cite{kalton--2003}.

In the following we discuss the existence of such barycenter maps  in the metric setting. Let \(\mathcal{P}_1(X)\) denote the set of all Radon probability measures \(\mu\colon \mathcal{B}(X)\to [0,1]\) on \(X\) such that 
\begin{equation}\label{eq:finite-first-moment}
\int_X d(x, x_0) \, \mu(dx) < \infty
\end{equation}
for some \(x_0\in X\) (and thus also for any other point in \(X\)). 
Notice that if \(\mu\), \(\nu\in \mathcal{P}_1(X)\), then \(\norm{\mu-\nu}_{\KR}<\infty\) and thus
\begin{equation}\label{eq:kr-duality}
W_1(\mu, \nu)=\norm{\mu-\nu}_{\KR}
\end{equation}
defines a metric on \(\mathcal{P}_1(X)\). This metric is mostly referred to as \(1\)-Wasserstein distance. However, the history of this distance (see, e.g.~\cite[pp. 106-108]{villani--2009}) shows that a more neutral name would be more appropriate, as it was repeatedly (re-)discovered by many different authors. Throughout this article, \(\mathcal{P}_1(X)\) will always be endowed with \(W_1\). Usually, the \(1\)-Wasserstein distance is defined in terms of optimal couplings of probability measures and \eqref{eq:kr-duality} is then a consequence of the Kantorovich--Rubinstein duality theorem. However, since we will have no use of this particular viewpoint we do not discuss this further. For more information on the Kantorovich-Rubinstein duality theorem, we refer the reader to \cite{edwards--2011}. The following definition is due to Sturm \cite[Remark 6.4]{sturm--2003}.

\begin{definition}
A map \(\beta\colon \mathcal{P}_1(X)\to X\) is called a barycenter map if it satisfies \(\beta(\delta_x)=x\) for all \(x\in X\) and is also 1-Lipschitz.
\end{definition}

It is easy to check that \(\delta\colon X \to \mathcal{P}_1(X)\) defined by \(x\mapsto \delta_x\) is an isometric embedding. Thus, the existence of a barycenter map \(\beta\colon \mathcal{P}_1(X) \to X\) is equivalent to saying that \(\delta\) admits a \(1\)-Lipschitz continuous left inverse. The following result is the cumulation of the works of Es-Sahib and Heinich \cite{es-sahib-1999}, Navas \cite{navas--2013}, and Descombes \cite{descombes--2016}.

\begin{theorem}\label{thm:existence-of-barycenter-map}
Let \(X\) be a complete metric space. Then the following statements are equivalent:
\begin{enumerate}[itemsep=0.75em]
\item \(X\) admits a barycenter map \(\beta\colon \mathcal{P}_1(X)\to X\);
\item \(X\) is a space of generalized non-positive curvature.
\end{enumerate}
\end{theorem}

We refer to \cite[Section 2]{basso2020extending} for a justification of how this result follows directly from \cite[Theorem 2.5]{descombes--2016}. See also \cite[Section 3.2]{basso--2018-neu}. Other barycentric constructions can be found in \cite{kirk--2014, espinola--2018}. For our proofs it would actually suffice to just have a barycenter map \(\beta\colon \mathcal{P}_f(X) \to X\) defined on the subset \(\mathcal{P}_f(X)\subset \mathcal{P}_1(X)\) of all \(\mu\in \mathcal{P}_1(X)\) whose support is finite. Distances between such measures can be estimated as follows.

\begin{lemma}\label{lem:barycentric-properties}
Assume that \(x_1, \ldots, x_n\in X\) are (not necessarily distinct) points of a metric space \(X\). Then for all \(n\)-tuples \((\alpha_1, \ldots, \alpha_n)\) and \((\beta_1, \ldots, \beta_n)\) in the standard \((n-1)\)-simplex \(\Delta^{n-1}\subset \R^{n}\),  we have
\begin{equation}\label{eq:espinola}
W_1(\sum_{i=1}^n \alpha_i \delta_{x_i}, \sum_{i=1}^n \beta_i \delta_{x_i}) \leq \frac{D}{2} \cdot \sum_{i=1}^n \abs{\alpha_i-\beta_i},
\end{equation}
where \(D=\max\{ d(x_i, x_j) : i, j=1, \ldots, n\}\).
\end{lemma}

We remark that if \(\mu\) and \(\nu\) denote the measures on the left-hand side of \eqref{eq:espinola}, the above sum is equal to the total variation \(\abs{\mu-\nu}(X)\). 

\begin{proof}[Proof of Lemma~\ref{lem:barycentric-properties}]
We set
\[
\mu=\frac{1}{n}\sum_{i=1}^n \alpha_i \delta_{x_i} \quad \text{ and } \quad \nu=\frac{1}{n}\sum_{i=1}^n \beta_i \delta_{y_i}.
\]
To begin, we suppose that the weights are rational numbers. Hence, there exist points \(y_1, z_1, \ldots, y_N, z_N \in X\) such that
\[
\mu=\frac{1}{N}\sum_{i=1}^N \delta_{y_i} \quad \text{ and } \quad \nu=\frac{1}{N}\sum_{i=1}^N \delta_{z_i}.
\]
It is well known that in this case
\begin{equation}\label{eq:explicit-wasserstein-distance}
W_1(\mu, \nu)= \min_{\pi\in S_{N}} \frac{1}{N} \sum_{i=1}^{N} d(y_i, z_{\pi(i)}) 
\end{equation}
where \(S_{N}\) denotes the symmetric group of order \(N\), see e.g. \cite[p. 56]{villani--2003}. Now, let \(\pi\in S_N\) be a permutation defined inductively as follows. Set \(A_0=\varnothing\) and abbreviate \(\Omega=\{1, \ldots, N\}\). Suppose \(\pi\) is defined on \(A_{k}\). We extend \(\pi\) to \(A_{k+1}\) according to the following rule. If there exist \(i\in \Omega \setminus A_k\) and \(j\in \Omega \setminus \pi(A_k)\) such that \(d(y_i, z_j)=0\), then we set \(A_{k+1}=A_{k} \cup \{i\}\) and \(\pi(i)=j\). If no such \(i\) and \(j\) exist, then we set \(A_{k+1}=\Omega\) and extend \(\pi\) to a permutation of \(\Omega\). Hence, by our construction of \(\pi\),
\begin{equation}\label{eq:with-special-minimizer}
\frac{1}{N} \sum_{i=1}^{N} d(y_i, z_{\pi(i)})\leq D \cdot \sum_{i\in I} \alpha_i-\beta_i,
\end{equation}
where \(I\) denotes the set of all \(i\in \{1, \ldots, n\}\) such that \(\alpha_i \geq \beta_i\). Let \(J\) denote the set of all \(j\in \{1, \ldots, n\}\) such that \(\beta_j \geq \alpha_j\). Clearly, it holds that
\[
\Delta:=\sum_{i\in I} \alpha_i-\beta_i=\sum_{j\in J} \beta_j-\alpha_j.
\]
and the right hand side of \eqref{eq:espinola} is equal to \(D\cdot \Delta\). Thus, by combining \eqref{eq:with-special-minimizer} with \eqref{eq:explicit-wasserstein-distance}, we find that  \eqref{eq:espinola} follows if \(\mu\) and \(\nu\) have rational weights. The general case now follows from this case, since according to the formula on \cite[p. 106]{villani--2009}, any \(\mu= \alpha_1 \delta_{x_1}+\dotsm+\alpha_n \delta_{x_n}\) can be approximated in \(1\)-Wasserstein distance with arbitrary precision by \(\mu'_k=\alpha_{1, k}' \delta_{x_{1}}+\dotsm+\alpha_{n, k}' \delta_{x_n}\) with rational weights in such a way that \(\alpha_{i, k}\to \alpha_i\) as \(k\to \infty\).
\end{proof}

\section{Whitney extensions}\label{sec:whitney}

In this section we prove Theorem~\ref{thm:Lang-Schlichenmaier-explicit} stated in the introduction. The proof of this theorem serves as a blueprint for the more complicated Lipschitz extension results we will prove later on. All of the most important techniques are already present and are more accessible as the more technical arguments in the subsequent sections. The main idea behind the proof of Theorem~\ref{thm:Lang-Schlichenmaier-explicit} is a variant of the classical argument that Whitney used in proving his extension theorem \cite{whitney--1934}. The extension is constructed by the help of a barycenter map in the target space and a Whitney-type covering of \(X\setminus A\). The existence of such a covering is guaranteed by the following proposition. 

\begin{proposition}\label{prop:Theorem5.2Revisited}
Let \(X\) be a metric space and \(A\subset X\) a closed subset satisfying \(\Nagata(n,c)\). 
Let \(r>1\) and define \(\varepsilon= (r-1)\tfrac{1}{2r}\) and \(\delta= \tfrac{\epsilon}{2r}\). Then there exists a covering \((B_i)_{i\in I}\) of \(X\setminus A\) such that
\begin{enumerate}[itemsep=0.75em]
\item\label{it:one} for all \(i\in I\), one has 
\[
\diam{B_i} \leq \alpha\cdot r_i,
\]
where \(r_i=d(B_i, A)\) and
\[\alpha= \frac{2(r+\varepsilon)}{1-\varepsilon}(1+c);
\]
\item\label{it:two} for every \(x\in X \setminus A\)
\[
x\in N_{\delta \cdot r_i}(B_i)
\]
for at most \(3(n+1)\) indices \(i\in I\);
\item\label{it:three} for all \(i\in I\),
\[
\hd(B_i, A)\leq (r+\varepsilon)\,r_i,
\]
where \(\hd(B_i, A)\) denotes the asymmetric Hausdorff distance from \(B_i\) to \(A\), that is, \(\hd(B_i, A)= \sup\{ d(b, A) : b\in B_i\}\).
\end{enumerate}   
\end{proposition}

The conditions above are called \textit{controlled diameter}, \textit{bounded multiplicity}, and \textit{controlled distance to $A$}, respectively. Proposition~\ref{prop:Theorem5.2Revisited} is a special case of a more refined construction from \cite[p. 3653]{lang-schlichenmaier}. Its proof can be found at the end of this section. 

Suppose now that we are in the setting of Theorem~\ref{thm:Lang-Schlichenmaier-explicit}. In particular, \(X\) is a metric space and \(A\subset X\) a closed subset satisfying \(\Nagata(n, c)\). We are given a \(1\)-Lipschitz map \(f\colon A \to Y\) to a complete metric space \(Y\) of generalized non-positive curvature. Notice that due to Theorem~\ref{thm:existence-of-barycenter-map}, we know that \(Y\) admits a barycenter map \(\beta\colon \mathcal{P}_1(Y)\to Y\). In the following, we construct a Lipschitz extension \(F\colon X \to Y\) of \(f\) using this barycenter map. 

Let \((B_i)_{i\in I}\) be a covering of \(X\setminus A\) as in Proposition~\ref{prop:Theorem5.2Revisited}. We define \(U_i= N_{\delta \cdot r_i}(B_i)\) for all \(i\in I\). We consider extensions \(F\colon X\to Y\) of the following form
\[
F(x)=\beta\big(\sum_{i\in I} \phi_i(x) \delta_{f(a_i)}\big),
\]
where \((\phi_i)\) is a partition of unity subordinated to \((U_i)\) and \((a_i)\) are suitably chosen points of \(A\). The main idea is that the partition of unity and the  points \((a_i)\) are compatible in the following sense. Whenever \(\phi_i(x)\), \(\phi_j(x) >0\), then we have \(d(a_i, a_j)\lesssim d(x, A)\). On the other hand, we want a good partition of unity in the sense that   \(\sum \Lip\phi_i(x)\lesssim d(x, A)^{-1}\). By combining these two properties, we can show that \(\Lip F(x) \leq L\) for some uniform constant \(L\). This then directly implies that \(F\) is \(L\)-Lipschitz, due to Lemma~\ref{lem:great-simplification}. In the following lemma we establish the existence of such a partition of unity, using an idea going back to Johnson, Lindenstrauss, and Schechtman; see \cite[p. 135]{lindenstrauss--1986}.

\begin{lemma}\label{lem:partition-of-unity}
There exists a Lipschitz partition of unity \((\phi_i)_{i\in I}\) subordinated to \((U_i)_{i\in I}\), such that for all \(x\in X\setminus A\),
\[
\sum_{i\in I} \Lip \phi_i (x) \leq 6 \cdot\frac{\log\bigl( 3(n+1)\bigr)}{\delta r_j},
\]
where \(j\in I\) is some index such that \(x\in B_{j}\). 
\end{lemma}

\begin{proof}
Set \(\psi_i(x)=d(x, X\setminus U_i)^{\,m}\) for some \(m>1\) to be determined later. Because of Proposition~\ref{prop:Theorem5.2Revisited}\eqref{it:two} and since \((U_i)_{i\in I}\) covers \(X\setminus A\), it follows that \(\psi(x)=\sum_{i\in I} \psi_i(x)\) is a well-defined positive function on \(X\setminus A\). We define \(\phi_i(x)= \psi_i(x)/\psi(x)\).  
A straightforward computation shows that
\[
\abs{\phi_i(x)-\phi_i(x')}\leq \frac{1}{\psi(x)}\abs{\psi_i(x)-\psi_i(x')}+\frac{\psi_i(x')}{\psi(x) \psi(x')}\sum_{j\in I} \abs{\psi_j(x')-\psi_j(x)},
\]
and so 
\[
\sum_{i\in I} \Lip \phi_i (x) \leq \frac{2}{\psi(x)} \sum_{i\in I} \Lip \psi_i(x).
\]
Clearly, 
\[
\Lip \psi_i(x)=m \cdot \psi_i(x)^{\frac{m-1}{m}}.
\]
Thus, by invoking H\"olders inequality with \(p=m\) and \(q=\tfrac{m}{m-1}\), we get
\[
\sum_{i\in I} \Lip \phi_i (x)\leq 2 m \cdot\Bigl(\frac{ 3(n+1)}{\psi(x)}\Bigr)^{\frac{1}{m}}.
\]
Since \(x\in B_j\) for some \(j\in I\), we have \(\psi(x)\geq d(x, U_j^c)^m\geq (\delta r_j)^m\). Notice that \(\log(3(n+1))\geq \log(3)>1\). Therefore, letting \(m=\log(3(n+1))\), we obtain the desired upper bound. 
\end{proof}

\begin{proof}[Proof of Theorem~\ref{thm:Lang-Schlichenmaier-explicit}]
Let \(X\) be a metric space, \(A\subset X\) a closed subset and \(f\colon A\to Y\) a \(1\)-Lipschitz map to a complete metric space \(Y\) of generalized non-positive curvature. Since \(X\) can  be embedded isometrically into \(\ell_\infty(X)\) by means of Kuratowski's embedding, we may assume without loss of generality that \(X\) is a Banach space. 

Fix \(r>1\) and \(\varepsilon_0>0\). Let \((B_i)_{i\in I}\) be a covering of \(X\setminus A\) as in Proposition~\ref{prop:Theorem5.2Revisited} and let \((\phi_i)_{i\in I}\) be as in Lemma~\ref{lem:partition-of-unity}. For each \(i\in I\), we select \(a_i\in A\) such that \(d(a_i, B_i) \leq (1+\varepsilon_0) r_i\). Moreover, we fix a barycenter map \(\beta\colon \mathcal{P}_1(Y)\to Y\). Now, we have everything at hand to define \(F\colon X \to Y\). We set \(F(a)= f(a)\) for all \(a\in A\), and 
\[
F(x)=  \beta\big(\sum_{i\in I}\phi_i(x) \delta_{f(a_i)}\big)
\]
for all \(x\in X\setminus A\). Let us first check that this defines a continuous extension of \(f\). As each function \(\phi_i\) is continuous, it easily follows that \(F\) is continuous at every point \(x\in X\setminus A\). Thus, we only need the check continuity at points \(a\in A\). Observe that
\[
d(F(a), F(x)) \leq \sum_{i\in I(x)} \phi_i(x) d(a, a_i),
\]
where we use the notation \(I(x)= \{ i\in I : \phi_i(x)>0\}\). Using the definition of \(a_i\), the controlled diameter condition of Proposition~\ref{prop:Theorem5.2Revisited}, and the definition of \(U_i\), we obtain that
\begin{equation}\label{eq:aux-1}
d(a_i, x) \leq \bigl( 1+\varepsilon_0+\alpha+ \delta\bigr) r_i
\end{equation}
for every \(i\in I(x)\). Moreover, \((1-\delta)\cdot r_i \leq d(x, A)\) for any such index \(i\). Hence, it follows from the above, that
\[
d(F(a), F(x))\leq \frac{2+\epsilon_0+\alpha+\delta}{1-\delta} \cdot d(a, x)
\]
for all \(a\in A\) and all \(x\in X\setminus A\). In particular, \(F\) is continuous at every point of \(A\). We now establish the existence of a constant \(L\geq 1\) such that \(\Lip F(x) \leq L\) for all \(x\in X\setminus A\). According to Lemma~\ref{lem:great-simplification} this then directly implies that \(F\) is \(L\)-Lipschitz, since \(X\) is a Banach space. With the help of Lemma~\ref{lem:barycentric-properties}, we compute for all \(x\), \(x'\in X\setminus A\),
\begin{align*}
d(F(x),F(x')) \leq \frac{D(x, x')}{2}\sum_{i\in I(x)\cup I(x')} \abs{\phi_i(x)-\phi_i(x')},
\end{align*}
where \(D(x, x')=\max\{ d(a_i, a_j) : i, j\in I(x)\cup I(x')\}\). Therefore,
\begin{equation}\label{eq:crucial-estimate}
\Lip F(x)\leq \frac{1}{2} \big[\limsup_{x' \to x} D(x, x') \big] \cdot \sum_{i\in I} \Lip \phi_i (x)  
\end{equation}
for all \(x\in X\setminus A\). Notice that \(I(x)\subset I(x')\) whenever \(x'\) is sufficiently close to \(x\). So \eqref{eq:aux-1} implies the following estimate,
\[
D(x, x') \leq 2\bigl( 1+\varepsilon_0+\alpha+ \delta)\cdot \max_{i\in I(x')} r_i. 
\]
Now, let \(j\in I(x)\) be such that \(x\in B_j\). Then using that for every \(i\in I(x')\),
\[
(1-\delta) \cdot r_i \leq d(x', A)\leq d(x, x')+d(x, A),
\]
as well as, \(d(x, A) \leq \hd(B_j, A)\leq (r+\varepsilon) r_j\),
we find that
\[
\max_{i\in I(x')} r_i \leq \frac{r+\varepsilon}{1-\delta} r_j+\frac{1}{1-\delta} d(x, x'). 
\]
Hence, by combining \eqref{eq:crucial-estimate} with Lemma~\ref{lem:partition-of-unity}, we conclude 
\begin{equation}\label{eq:what-we-want}
\Lip F(x) \leq  6\,\bigl( 1+\varepsilon_0+\alpha+ \delta) \frac{r+\varepsilon}{\delta(1-\delta)}  \log(3(n+1)). 
\end{equation}
We set \(r=\tfrac{5}{4}\). Then \(\varepsilon=\tfrac{1}{10}\), \(\delta=\tfrac{1}{25}\), \(\alpha= 3(1+c)\) and so \eqref{eq:what-we-want} evaluates to
\[
\Lip F(x) \leq \frac{3375}{16} \cdot\big[3 (1 + c) +\frac{26}{25}\big]\cdot  \log(3(n+1))+C_0 \varepsilon_0,
\]
for some \(C_0\) independent of \(\varepsilon_0\). 
Since \(\log(3(n+1))\leq \tfrac{23}{20} \cdot \log_2(n+2)\), the desired estimate \eqref{eq:estimate-of-F} follows. 
\end{proof}

\begin{proof}[Proof of Proposition~\ref{prop:Theorem5.2Revisited}]
Fix \(k\in \Z\). We define
\[R=R_k= \bigl\{ x\in X \,:\, r^k \leq d(x, A) < r^{k+1} \bigr\}.\]
Let \(W\subset R\) be a maximal (with respect to inclusion) \(\varepsilon r^k\)-separated subset of \(R\) and \(\rho \colon W\to A\) any map such that \(d(\rho(w), w)\leq r^{k+1}\) for all \(w\in W\). To ensure the existence of such a map we rely on the axiom of choice. Let
\[
s= 2\bigl(2\varepsilon+r\bigr) r^{k}
\]
and assume \((A_i)_{i\in I_k}\) is a \(cs\)-bounded cover of \(A\) with \(s\)-multiplicity at most \(n+1\). We define 
\[
B_i(\eta)= \bigl\{ x\in X\setminus A : \exists  w\in W \text{ s. t. } d(x,w)\leq \varepsilon r^{k}+\eta \text{ and } \rho(w)\in A_i \bigr\}
\]
and \(B_i= B_i(0)\) for every \(i\in I_k\). By construction, \(\mathcal{B}_k=(B_i)_{i\in I_k}\) is a covering of \(R\) by subsets of \(X\setminus A\). Clearly,
\[
\diam{B_i}\leq 2(\varepsilon+r) r^{k}+cs \leq \bigl( 2(1+c)\bigl(\varepsilon+r\bigr)\bigr) r^k,
\]
and so \(\diam{B_i} \leq \alpha \cdot d(B_i, A)\) for all \(i\in I_k\). Moreover, for every \(b\in B_i\),
\[
d(b, A)\leq (\epsilon+r) r^k.
\]
Let \(\mathcal{B}\) be the union of all \(\mathcal{B}_k\). This is a covering of \(X\setminus A\)  and by the above, \eqref{it:one} and \eqref{it:three} follow. To finish the proof we show \eqref{it:two}, the bounded multiplicity condition. It is easy to check that \(\delta \cdot d(B_i, A) \leq \varepsilon r^k\) for every \(i\in I_k\) and therefore
\begin{equation}\label{eq:reduction1}
N_{\delta r_i}(B_i) \subset B_i(\varepsilon r^{k}).
\end{equation}
Fix \(x\in X\setminus A\). We claim that \(x\) meets at most \(n+1\) members of \(\mathcal{B}_k^+=(B_i(\varepsilon r^{k}))_{i\in I_k}\).
For every \(B_i(\varepsilon r^{k})\) that meets \(x\) there exists \(w_i\in W\) such that \(d(w_i, x)\leq 2 \varepsilon r^{k}\) and \(\rho(w_i)\in A_i\). We consider \(M= \bigl\{ \rho(w_i) : B_i(\varepsilon r^{k}) \cap \{x\} \neq \varnothing \bigr\}.\)  Since
\[
\diam(M) \leq 2\bigl(2\varepsilon+r\bigr) r^{k}=s, 
\]
it follows that \(M\) meets at most \(n+1\) members of \((A_i)_{i\in I_k}\) and consequently \(x\) meets at most \(n+1\) sets of \(\mathcal{B}_k^+\). Now, let \(\mathcal{B}^+\) be the union of all \(\mathcal{B}_k^+\). Since each \(B_i(\varepsilon r^{k})\) is contained in \(R_{k-1}\cup R_{k}\cup R_{k+1}\), it follows from the above, that \(x\) meets at most \(3(n+1)\) members of \(\mathcal{B}^+\).  Therefore, using \eqref{eq:reduction1}, we conclude that there are at most \(3(n+1)\) indices \(i\) in \(I= \bigcup I_k\) such that \(x\in N_{\delta r_i}(B_i)\), as desired.
\end{proof}

\section{Multiscale extensions}\label{sec:multiscale}

Throughout this section, let \(X\) be a Banach space and \(A\subset X\) a subset consisting of at most \(n\) points for some \(n\geq 3\). Suppose we are given a \(1\)-Lipschitz map \(f\colon A \to Y\) to a Banach space \(Y\). Our aim is to construct an extension \(F\colon X \to Y\) of \(f\) whose Lipschitz constant is as small as possible. For this we make the following “ansatz" for \(F\):
\[
F(x)=\frac{1}{N+1}\sum_{n\in \Z} \omega_n(x) F_n(x)
\]
where \(N\) is a sufficiently large positive integer, \(\omega_n \colon X \setminus A \to \R\) are suitably chosen cutoff functions, and  \(F_n \colon X_n \to Y\) are continuous maps on 
\[
X_n=\{ x\in X : d(x, A) \leq 2^n\}.
\]
The main reason why the proof works is that we can use Lemma~\ref{lem:itertive-ball-partitioning} to construct \(F_n\)'s in such a way that
\[
\sum \Lip F_n(x) \lesssim \log(n)
\]
where the sum runs over all \(n\) with \(\omega_n(x)=1\). This then implies that
\[
\Lip F(x) \lesssim \frac{\log(n)}{N}+\frac{2^N}{N},
\]
where the second term comes from those \(n\) for which \(\omega_n(x) \neq 0, 1\). Setting \(N\approx \log(\log(n))\) then induces the desired bound \(\log(n)/ \log(\log(n))\) in Lee and Naor's extension theorem.

We proceed with the construction of the maps \(F_n\). For every \(x\in X\), select \(a_x\in A\) such that 
\begin{equation}\label{eq:distance-to-ax}
d(x, a_x)=d(x, A),
\end{equation}
since \(A\) is finite such a point surely exists. 

\begin{lemma}\label{lem:construction-F_n}
For every \(n\in \Z\) there exists a map \(F_n\colon X_n \to Y\) such that \(\norm{F_n(x)-f(a_x)}\leq 2^n\) for all \(x\in X_n\) and 
\begin{equation}\label{eq:pointwise-bound-Lipschitz-constant-F_n}
\Lip F_n(x) \leq 120\cdot \log\big( \frac{\# B(a_x, 2^n)}{\# B(a_x, 2^{n-2})}\big)
\end{equation}
for all \(x\in X_n\) with \(d(x, A) \leq \frac{1}{16}\cdot 2^n\).
\end{lemma}

In order to move quickly to the proof of Lee and Naor's theorem, we postpone the proof of Lemma~\ref{lem:construction-F_n} to the end of the section. Let us now discuss how to construct suitable cutoff functions \(\omega_n\).  

Let \(\omega\colon \R\to \R\) be the unique piecewise linear function such that \(\omega(t)=1\) on \([1, 2^N]\) and \(\omega(t)=0\) if \(t\notin (2^{-1}, 2^{N+1})\). We set
\[
\omega_n(x)= \omega\big(\frac{1}{16}\cdot \frac{ 2^n}{d(x, A)}\big)
\]
for all \(n\in \Z\) and all \(x\in X\setminus A\). The following lemma collects all relevant properties of the \(\omega_n\)'s.

\begin{lemma}\label{lem:properties-of-cutoff-functions}
For every \(x\in X\setminus A\), one has
\begin{equation}\label{eq:miracle}
\sum_{n\in \Z} \omega_n(x)=N+1.
\end{equation}
 Moreover, if \(x\in X\setminus A\) is fixed and \(n_0\) denotes the maximal integer such that \(\frac{1}{16}\cdot 2^{n_0} \leq d(x, A)\), then  
\[
\begin{cases}
\omega_n(x)\in [0,1] & \text{ if \(n=n_0\) or \(n=n_0+(N+1)\)} \\
\omega_n(x)=1 & \text{ if \(n=n_0+1,\, \ldots,\, n_0+N\)} \\
\omega_n(x)=0 & \text{ otherwise,}
\end{cases}
\]
as well as 
\begin{equation}\label{eq:estimate-on-Lip-Omega}
\Lip \omega_n(x) \leq 2 \cdot \frac{1}{d(x, A)}
\end{equation}
if \(n\) is contained in \(\{n_0-1,\, n_0,\, n_0+N,\, n_0+(N+1)\}\) and \(\Lip \omega_n(x)=0\) otherwise.
\end{lemma}

\begin{proof}
To begin, we show \eqref{eq:miracle}. Letting \(\alpha=\frac{1}{16}\cdot \frac{2^{n_0}}{d(x, A)}\), we find that 
\[
\sum_{n\in \Z} \omega_n(x)=\sum_{n\in \Z} \omega( \alpha \cdot 2^n).
\]
Since \(1/2 <\alpha \leq 1\), it follows from the definition of \(\omega\) that
\[
\sum_{n\in \Z} \omega_n(x)=\sum_{n=0}^{N+1} \omega(\alpha \cdot 2^n)=\omega(\alpha\cdot 2^0)+N+\omega(\alpha\cdot 2^{N+1}).
\]
Notice that \(\omega(t)=2t-1\) on \((2^{-1}, 1]\) and \(\omega(t)=-\frac{1}{2^N}t+2\) on \((2^N, 2^{N+1}]\). Therefore,
\[
\omega(\alpha)+\omega(\alpha\cdot 2^{N+1})=1,
\]
and so \eqref{eq:miracle} follows. To finish the proof, we establish \eqref{eq:estimate-on-Lip-Omega}. We have \(\Lip \omega(t)=2\) on \([2^{-1}, 1]\), \(\Lip \omega(t)=2^{-N}\) on \([2^N, 2^{N+1}]\) and \(\Lip \omega(t)=0\) otherwise. Hence, we only need to consider integers \(n\)  for which \(\frac{1}{16}\cdot \frac{2^{n}}{d(x, A)}\) is contained in 
\[
[2^{-1}, 1]\quad  \text{ or }\quad  [2^N, 2^{N+1}].
\]
This is can only happen if \(n\in \{n_0-1,\, n_0,\, n_0+N,\, n_0+(N+1)\}\). The indices \(n_0-1\) and \(n_0+N\) must be included to account for the case where \(d(x, A)=\frac{1}{16}\cdot 2^{n_0}\). We abbreviate \(t=d(x, A)\). Notice that 
\begin{equation}\label{eq:pointwise-omega}
\Lip \omega_n(x)\leq \Lip f_n(t), 
\end{equation}
where \(f_n\colon \R_{>0} \to \R_{>0}\) is defined by \(f_n(s)=\omega( c_n \cdot s^{-1})\) with \(c_n=\frac{1}{16}\cdot 2^{n}\). We have
\[
\Lip f_n(t) \leq 
\begin{cases}
2 c_n/ t^2, & \text{ if \(n=n_0-1\), \(n_0\)} \\
c_n/ (2^{N} t^2) & \text{ if \(n=n_0+N\), \(n_0+(N+1)\)}
\end{cases}
\]
and so \eqref{eq:estimate-on-Lip-Omega} follows from \eqref{eq:pointwise-omega}, since \(c_{n_0}/t \leq 1\).
\end{proof}

We now already have everything at hand to prove Lee and Naor's extension theorem stated in the introduction.

\begin{proof}[Proof of Theorem~\ref{thm:lee-and-naor}]
A straightforward application of \eqref{eq:miracle} shows that 
\[
\norm{F(x)-f(a_x)}\leq \frac{1}{N+1}\sum_{n\in \Z} \abs{\omega_n(x)}\cdot \norm{F_n(x)-f(a_x)}
\]
for all \(x\in X \setminus A\). Since \(\omega_n(x)\neq 0\) for at most \(N+2\) indices, we find by Lemma~\ref{lem:construction-F_n} that
\[
\norm{F(x)-f(a_x)} \leq 2^N \cdot d(x, A).
\]
Thus, to finish the proof, because of Lemma~\ref{lem:great-simplification}, it suffices to show that for all \(x\in X\setminus A\),
\[
\Lip F(x) \leq L.
\]
To this end,  we fix \(x\in X \setminus A\) and let \(n_0\) be defined as in Lemma~\ref{lem:properties-of-cutoff-functions}. Let \(x'\in X \setminus A\) be sufficiently close to \(x\) and denote by \(\Omega\subset \Z\) the set consisting of those indices \(n\) for which \(\omega_n(x)\neq 0\) or \(\omega_n(x')\neq 0\). Notice that \(\Omega=\{n_0-1,\, n_0,\, \ldots, \, n_0+N,\, n_0+(N+1)\}\). Lemma~\ref{lem:properties-of-cutoff-functions} tells us that \(\Lip \omega_n(x)=0\) if \(n\) is not equal to 
\[
a=n_0-1,\quad b=n_0,\quad c=n_0+N,\quad d=n_0+(N+1),
\]
and 
\[
\Lip \omega_n(x) \leq 2 \cdot \frac{1}{d(x, A)}<32 \cdot 2^{-n_0}
\]
if \(n=a\), \(b\), \(c\), \(d\). For each \(n\in \Z\), we consider the shifted map \(\widetilde{F}_n=F_n-v\) for \(v=f(a_x)\). Clearly, because of \eqref{eq:miracle}, we have
\[
F(x)-F(x')=\frac{1}{N+1}\cdot \sum_{n\in \Omega} \big[\omega_n(x) \widetilde{F}_n(x)- \omega_n(x') \widetilde{F}_n(x')\big].
\]
We now analyse each term
\[
T_n=\norm{\omega_n(x) \widetilde{F}_n(x)- \omega_n(x') \widetilde{F}_n(x')}
\]
separately. If \(n\in \Omega\setminus \{a, b, c, d\}\) then
\[
T_n=\norm{F_n(x)-F_n(x')}
\]
and if \(n=a\), \(b\), \(c\), \(d\) then
\[
T_n \leq \omega_n(x')\cdot \norm{F_n(x)-F_n(x')}+\abs{\omega_n(x)-\omega_n(x')}\cdot \norm{\widetilde{F}_n(x)}.
\]
Since \(\omega_n(x')\leq 1\), this implies that
\[
\Lip F(x) \leq \frac{1}{N+1} \Big[\sum_{n\in \Omega} \Lip F_n(x)\Big] +\frac{1}{N+1}\cdot \frac{128}{2^{n_0}}\cdot\Big[\max_{n=a, b, c, d} \, \norm{\widetilde{F}_n(x)}\Big].
\]
It therefore follows from Lemma~\ref{lem:construction-F_n} that
\begin{align*}
\Lip F(x) \leq \frac{120}{N+1} \Big[ \sum_{n=a}^{d} \log\Big( \frac{\# B(a_x, 2^n)}{\# B(a_x, 2^{n-2})}\Big)\Big]+\frac{1}{N+1}\cdot  \frac{128}{2^{n_0}}\cdot 2^{n_0+(N+1)},
\end{align*}
which simplifies to
\[
\Lip F(x) \leq 240 \cdot \frac{\log(n)}{N+1}+128 \cdot \frac{2^{N+1}}{N+1}.
\]
Hence, setting \(N+1=\lfloor \log_2(\log(N+1)) \rfloor\), we get
\[
\Lip F(x) \leq 600 \cdot \frac{\log(n)}{\log(\log(n))},
\]
as desired.
\end{proof}

We finish this section with the proof of Lemma~\ref{lem:construction-F_n}, which ensures the existence of the maps \(F_n\) used above.

\begin{proof}[Proof of Lemma~\ref{lem:construction-F_n}]
We will construct \(F_n\colon X_n \to Y\) using a Whitney type argument that uses a covering of \(X_n\) consisting of only finitely many sets.
Fix \(n\in \Z\) and let \((A_i)_{i\in I}\) be a \(2^{n}\)-bounded covering of \(A\) as in Lemma~\ref{lem:itertive-ball-partitioning}. In particular, \(I\) is finite and for every \(a\in A\),
\begin{equation}\label{eq:ratio-of-multiplicities}
\frac{\#\{ i\in I : B(a, 2^{n-2}) \subset A_i\}}{\#\{ i\in I : a\in A_i\}} \geq \frac{\# B(a, 2^{n-2})}{\# B(a, 2^n)}.
\end{equation}
For every \(i\in I\) we set \(B_i=\{ x\in X_n : a_x\in A_i\}\). Clearly, \((B_i)_{i\in I}\) defines a covering of \(X_n\).

In what follows, we fix a point \(x\in X_n\) with \(d(x, A) \leq \frac{1}{16} \cdot 2^n\). For simplicity, we also abbreviate \(\delta=\frac{1}{16}\). Suppose that \(x'\in X_n\) satisfies \(d(x, x') \leq \delta \cdot 2^n \). Then, because of \eqref{eq:distance-to-ax},
\[
\norm{a_x- a_{x'}} \leq d(x, A)+d(x, x')+d(x', A) \leq 2^{n-2},
\]
and so 
\begin{equation}\label{eq:multiplicity-of-B_i}
\#\big\{ i\in I : B(x, \delta \cdot 2^{n}) \subset B_i\big\} \geq \#\big\{ i\in I : B(a_x, 2^{n-2}) \subset A_i\big\}. 
\end{equation}
We now construct a partition of unity \((\phi_i)_{i\in I}\) using the same strategy as in Lemma~\ref{lem:partition-of-unity}. Let \(m>1\) be a real number, to be determined later. We set \(\psi_i(x)=d(x, X_n \setminus B_i)^m\). Because of \eqref{eq:multiplicity-of-B_i}, it follows that \(\psi(x)=\sum_{i\in I} \psi_i(x)\) is a positive function on \(X_n\). Letting \(\phi_i=\psi_i(x)/ \psi(x)\), we find by exactly the same reasoning as in the proof of Lemma~\ref{lem:partition-of-unity}, that
\begin{equation}\label{eq:desired-estimate-on-Lip}
\sum_{i\in I} \Lip \phi_i(x)\leq 2m\cdot \Big( \frac{\#\{ i\in I : \psi_i(x)\neq 0 \}}{\psi(x)} \Big)^{\frac{1}{m}}.
\end{equation}
We have \(\psi_i(x)\neq 0\) if and only if \(x\) is contained in the interior of \(B_i\). Hence, using that \(x\in B_i\) if and only if \(a_x\in A_i\), we obtain
\[
\#\big\{ i\in I : \psi_i(x)\neq 0 \big\}\, \leq \, \#\big\{ i\in I : a_x\in A_i\big\}.
\]
Since
\[
\psi(x)\geq (\delta\cdot 2^n)^m \cdot\Big[ \#\big\{ i\in I : B(x, \delta\cdot 2^{n}) \subset B_i\big\}\Big],
\]
it follows from \eqref{eq:desired-estimate-on-Lip} with the help of \eqref{eq:multiplicity-of-B_i} and \eqref{eq:ratio-of-multiplicities} that
\[
\sum_{i\in I} \Lip \phi_i(x) \leq \frac{2}{\delta \cdot 2^n} \cdot\Big[ m\cdot \Big( \frac{\# B(a_x, 2^{n})}{\# B(a_x, 2^{n-2})} \Big)^{\frac{1}{m}}\Big]. 
\]
Thus, setting \(m\) equal to the logarithm of the term from which the \(m\)-th root is taken in the expression above, we obtain
\begin{equation}\label{eq:inequality-of-pointtwise-Lipschitzconstant-of-phi-i}
\sum_{i\in I} \Lip \phi_i(x) \leq 100 \cdot 2^{-n} \cdot \log\Big( \frac{\# B(a_x, 2^n)}{\# B(a_x, 2^{n-2})}\Big)
\end{equation}
for each \(x\in X_n\) with \(d(x, A) \leq \frac{1}{16} \cdot 2^n\). Having such a partition of unity \((\phi_i)_{i\in I}\) at hand it is now not difficult to construct a map \(F_n\colon X_n \to Y\) with the desired properties. For each \(i\in I\) select \(a_i \in A_i\) and for every \(x\in X_n\) we set
\[
F_n(x)=\sum_{i\in I} \phi_i(x) f(a_i)
\]
Notice that \(F_n\) is not necessarily an extension of \(f\). But for every \(x\in X_n\), we have
\[
\norm{F_n(x)-f(a_x)}\leq \sum_{i\in I} \phi_i(x) \cdot \norm{a_i-a_x}
\]
where we used that \(f\) is \(1\)-Lipschitz. Since \(a_x\in A_i\) whenever \(\phi_i(x)\neq 0\), it follows that \(\norm{F_n(x)-f(a_x)}\leq 2^n\), as desired. To finish the proof we show \eqref{eq:pointwise-bound-Lipschitz-constant-F_n}.  
Fix \(x\in X_n\) with \(d(x, A)\leq \frac{1}{16}\cdot 2^n\). For every \(x'\in X_n\) one has
\[
F_n(x)-F_n(x')=\sum_{i \in I} \bigl[\phi_i(x)-\phi_i(x') \bigr]\cdot [ f(a_i)-f(a_x)]
\]
and thus 
\begin{equation}\label{eq:intermediate-step-pointwise-Lipschitz-constant-F-n}
\Lip F_n(x)\leq \sum_{i\in I} \Lip \phi_i(x) \cdot \norm{a_i- a_x}.
\end{equation}
Clearly,  if \(\Lip \phi_i(x) \neq 0\), then \(d(x, B_i)=0\). If \(x\) is contained in the closure of \(B_i\), then
\[
\norm{a_i-a_x}\leq 2^n+d(x, A)+d(x, x')+d(x', A)
\]
for any \(x'\in B_i\). This implies that \(\norm{a_i-a_x} \leq (1+\tfrac{1}{8})\cdot 2^n\) whenever \(\Lip \phi_i(x) \neq 0\). It thus follows from \eqref{eq:intermediate-step-pointwise-Lipschitz-constant-F-n} that
\[
\Lip F_n(x) \leq (1+\tfrac{1}{8})\cdot 2^n \cdot \Big[\sum_{i\in I} \Lip \phi_i(x)\Big]
\]
and so \eqref{eq:pointwise-bound-Lipschitz-constant-F_n} follows from \eqref{eq:inequality-of-pointtwise-Lipschitzconstant-of-phi-i}.
\end{proof}

\section{Geometry of simplicial complexes}\label{sec:geometry-of-simplicial-complexes}
All simplicial complexes appearing in this paper are interpreted as geometrical objects. For a given non-empty index set \(I\) we consider 
\[
\Sigma(I)=\big\{ x\in \ell_2(I) : x_i \geq 0 \text{ and } \sum_{i\in I} x_i=1 \big\}
\]
and equip it with the metric induced by the norm \(\abs{\,\cdot\,}\) of \(\ell_2(I)\). Spaces of this form will serve as the ambient space for all the simplicial complexes we consider.  We say that \(\Delta\subset \Sigma(I)\) is an \(n\)-simplex, for some \(n\geq 0\), if there is an \((n+1)\)-point subset \(J\subset I\) such that 
\begin{equation}
\Delta=\bigl\{ x\in\Sigma(I) : \sum_{j\in J} x_j=1\bigr\}.
\end{equation}
Notice that each \(n\)-simplex \(\Delta\subset\Sigma(I)\) is metrized with the Euclidean distance and isometric to the standard \(n\)-simplex \(\Delta^{n}\) whose definition we recall in \eqref{eq:simplex-def-2}. A subset \(\Delta\subset \Sigma(I)\) is said to be a simplex if it is an \(n\)-simplex for some \(n\geq 0\).

\begin{definition}
A simplicial complex in \(\Sigma(I)\) is a subset \(\Sigma\subset \Sigma(I)\) such that
\(\Sigma\) is the union of simplices \(\Delta_\alpha\subset \Sigma(I)\), \(\alpha\in A\). 
\end{definition}
Whenever we say that \(\Sigma\) is a simplicial complex, we tacitly assume that \(\Sigma \subset \Sigma(I)\) for some index set \(I\). We equip each simplicial complex \(\Sigma\subset \Sigma(I)\) with the induced metric. This metric will be referred to as the \(\ell_2\)-metric. Notice that whenever \(\Sigma\) contains more than one point then \(\diam(\Sigma)=\sqrt{2}\) with respect to the \(\ell_2\)-metric.

The following lemma will be used in the proof of Theorem~\ref{thm:bilipschitz-triangulation} at the end of this section, but it is also of independent interest. 
\begin{lemma}\label{lem:simplicial-geometry}
Let \(\Sigma\) be a simplicial complex equipped with the \(\ell_2\)-metric. If \(\Delta\), \(\Delta'\subset \Sigma\) are two \(n\)-simplices such that \(\Delta \cap \Delta'\neq \varnothing\), then for all \((x, y)\in \Delta \times \Delta'\), there exists \(z\in \Delta \cap \Delta'\) such that
\[
\abs{x-z}+\abs{z-y} \leq 4\cdot \sqrt{n}\cdot\abs{x-y}.
\]
\end{lemma}

\begin{proof}
We may suppose that \(\Delta\neq \Delta'\), as otherwise there is nothing to prove. Consequently, \(k=\dim(\Delta \cap \Delta')\) satisfies \(k\in \{0, \ldots, n-1\}\). Without loss of generality, we may suppose that \(x\), \(y\in \R^{2n+1-k}\) with
\[
x=(x_0, \ldots, x_k, x_{k+1}, \ldots, x_n, 0, \ldots, 0)
\]
and
\[
y=(y_0, \ldots, y_k, 0, \ldots,0, y_{k+1}, \ldots, y_n).
\]
We define
\[
z=(x_0+\nu, x_1, \ldots, x_k, 0, \ldots, 0)
\]
for 
\[
\nu=\sum_{i>k}^n x_i.
\]
Clearly, \(z\in \Delta \cap \Delta'\), as well as
\[
\abs{x-z}^2=\nu^2+\sum_{i>k}^n x_i^2,
\]
\[
\abs{y-z}^2=(y_0-z_0)^2+\sum_{i=1}^k (x_i-y_i)^2+\sum_{i>k}^n y_i^2.
\]
Since
\[
\nu^2=\big(\sum_{i>k}^n x_i\big)^2\leq (n-k)\sum_{i>k}^n x_i^2
\]
and
\[
(y_0-z_0)^2=((y_0-x_0)-\nu)^2\leq 2(x_0-y_0)^2+2 \nu^2,
\]
we find that
\[
\abs{x-z}^2+\abs{z-y}^2 \leq 2\sum_{i=0}^k (x_i-y_i)^2+(m+2m+1)\sum_{i>k}^n x_i^2+\sum_{i>k}^n y_i^2.
\]
for \(m=n-k\). Notice that \(1\leq m \leq n\). Since
\[
(\abs{x-z}+\abs{z-y})^2 \leq 2(\abs{x-z}^2+\abs{z-y}^2)
\]
and
\[
\abs{x-y}^2=\sum_{i=0}^k (x_i-y_i)^2+\sum_{i>k}^n x_i^2+\sum_{i>k}^n y_i^2,
\]
the desired estimate follows. (In fact, \(4\) could even be replaced by \(2\sqrt{2}\).)
\end{proof}

If a simplicial complex \(\Sigma\) is path-connected, then the associated length metric on \(\Sigma\) will be called the length metric.  It is not difficult to construct simplicial complexes that are unbounded with respect to the length metric.
Hence, in general, the length metric and the \(\ell_2\)-metric are clearly not bi-Lipschitz equivalent. However, by a straightforward use of the lemma above, we find that this is true whenever \(\Sigma\) is pure \(n\)-dimensional and consists of finitely many \(n\)-simplices. Before stating this explicitly let us briefly fix some terminology regarding the dimension of simplicial complexes.

For any \(n\geq 0\) we let \(\Sigma^{(n)}\) denote the \(n\)-skeleton of \(\Sigma\), that is, the union of all \(k\)-simplices \(\Delta\subset \Sigma\) with \(0\leq k\leq n\). We say that \(\Sigma\) is \(n\)-dimensional provided that \(\Sigma^{(n)}=\Sigma^{(n+1)}\). Moreover, \(\Sigma\) is called pure \(n\)-dimensional if every of its simplices is contained in some \(n\)-simplex of \(\Sigma\).

\begin{lemma}\label{lem:length-metric-vs-l2-metric}
Let \(\Sigma\) be a pure \(n\)-dimensional simplicial complex, for some \(n\geq 2\), such that the number of \(n\)-simplices is bounded by \(N\). Then \(\Sigma\) equipped with the \(\ell_2\)-metric is 
\[
N^{10\cdot \log(n)}
\]
quasiconvex whenever it is connected.
\end{lemma}

\begin{proof}
Let \(x\), \(y\in \Sigma\) and suppose \(\Sigma\) is connected. Then there exists a chain of \(n\)-simplices \(\Delta_0, \ldots, \Delta_m\subset \Sigma\) such that \(x\in \Delta_0\) and \(y\in \Delta_m\), and consecutive simplices meet. As we do not require the simplices to be distinct, we may suppose that \(m+1=2^{\ell}\) for \(\ell\geq 1\). We claim that for \(m'=2^{\ell-1}\) there exists \(z\in \Delta_{m'}\cap \Delta_{m'-1}\) such that
\begin{equation}\label{eq:inductive-step}
    \abs{x-z}+\abs{z-y} \leq c\cdot \abs{x-y},
\end{equation}
where \(c=4\cdot \sqrt{n}\). Thus, in particular, by applying this construction iteratively \(\ell\)-times,  we find that there exist \(x_i\in \Delta_i \cap \Delta_{i-1}\), for \(i=1, \ldots, m\), such that
\[
\sum_{i=1}^{2^{\ell}} \abs{x_{i-1}-x_i}\leq c \cdot\sum_{i=1}^{2^{\ell-1}}\abs{x_{2(i-1)}-x_{2i}} \leq \dotsm \leq c^\ell \cdot \abs{x_0-x_{m+1}},
\]
where \(x_0=x\) and \(x_{m+1}=y\). This directly implies that \(\Sigma\) is \(c^{3\cdot \log(N)}\)-quasiconvex, as desired. Notice that \eqref{eq:inductive-step} is true if \(m=1\) due to Lemma~\ref{lem:simplicial-geometry}. We now suppose \(m\geq 2\) and that \(\Delta_0\) and \(\Delta_{m}\) have no points in common. Without loss of generality, we may suppose that
\(x\), \(y\in \R^{2(n+1)+1}\) with
\[
x=(x^0, \ldots, x^n, 0, \ldots, 0, 0)
\]
and
\[
y=(0, \ldots,0, y^{0}, \ldots, y^n, 0),
\]
as well as that at least one vertex of \(\Delta_{m'}\cap \Delta_{m'-1}\) is contained in \(\R^{2(n+1)+1}\). Let \(z\in \Delta_{m'}\cap \Delta_{m'-1}\) be such a vertex.
We consider the triangle \(x z y\) in \(\R^{2(n+1)+1}\). Notice that if \(z\in \Delta_0\), then assuming that \(z=(1, 0, \ldots, 0)\) we have 
\begin{align*}
\cos\big( \angle \,x z y\big)&=\frac{\langle x-z, y-z \rangle}{\abs{x-z}\cdot \abs{y-z}} \\
&=\frac{1-x^0}{\Big((1-x^0)^2+\sum_{i>0} (x^i)^2\Big)^{1/2}}\cdot \frac{1}{\Big(1+\sum_{i\geq 0} (y^i)^2\Big)^{1/2}},
\end{align*}
which implies that
\[
\cos\big( \angle \,x z y\big) \leq \Big(1+\sum_{i\geq 0} (y^i)^2\Big)^{-1/2}\leq \Big(1-\frac{1}{n+2}\Big)^{1/2}.
\]
On the other hand, if \(z=(0,\ldots, 0, 1)\) (and thus is neither contained in \(\Delta_0\) nor in \(\Delta_m\)) an analogous computation yields the same upper bound
\[
\cos\big( \angle \,x z y\big) \leq \Big(1-\frac{1}{n+2}\Big)^{1/2}.
\]
Let \(\eta\), \(\kappa\) be positive real numbers such that \(\eta^2=1-1/(n+2)\) and \(\kappa^2=(1-\eta)/2\). Observe that  
\[
\kappa \leq \sqrt{\frac{1-\cos(\gamma)}{2}},
\]
where \(\gamma=\angle \,x z y\). Hence, a straightforward computation using the law of cosines reveals that
\[
\kappa \cdot \big( \abs{x-z}+\abs{z-y}\big) \leq \abs{x-y}.
\]
Since \(\kappa^{-1} \leq c\) this implies \eqref{eq:inductive-step}. 
\end{proof}

We conclude this section with the short proof of Theorem~\ref{thm:bilipschitz-triangulation} from the introduction.

\begin{proof}[Proof of Theorem~\ref{thm:bilipschitz-triangulation}]
Let \(A\subset X\) and suppose \(f\colon A \to Y\) is a \(1\)-Lipschitz map to a complete \(\CAT(0)\) space \(Y\). Our objective is to find a Lipschitz extension \(F\colon X\to Y\) of \(f\) which has the desired Lipschitz constant. By \cite[Corollary~4.3]{basso--wenger--young}, \(X\) and \(\Sigma\) equipped with the length metric are \(D\)-bi-Lipschitz equivalent. Moreover, because of Lemma~\ref{lem:length-metric-vs-l2-metric}, it follows that \(\Sigma\) with the length metric is \(N^{10\cdot \log(n)}\)-bi-Lipschitz equivalent to \(\Sigma\) equipped with the \(\ell_2\)-metric. Hence, according to \cite[Lemma~2.4]{basso--2022} it suffices to consider the case where \(A\subset \Sigma\) for \(\Sigma\) equipped with the \(\ell_2\)-metric. But in this case the generalized Kirszbraun theorem \cite[Theorem A]{lang--schroeder} of Lang and Schroeder implies that \(f\colon A \to Y\) admits a \(1\)-Lipschitz extension \(F\colon \Sigma\to Y\). 
\end{proof}

\section{A general Whitney-type extension theorem}\label{sec:general-whitney}
In this section we formulate a more axiomatic approach to Lang and Schlichenmaier's extension theorem. Outsourcing the difficult steps in the theorem as definitions yields a short and conceptually simple proof of a general Lipschitz extension theorem. We will rely on the following two key definitions.

\begin{definition}\label{def:simplicial-target-space}
A metric space \(Y\) is said to be an \((n, C)\)-simplicial extensor if the following holds. Whenever \(\Sigma\) is a simplicial complex of dimension at most \(n\) equipped  with the \(\ell_2\)-metric and \(f\colon \Sigma^{(0)}\to Y\) is a map, then \(f\) admits an extension \(F\colon \Sigma \to Y\) such that
\[
\Lip F|_{\Delta} \leq C \cdot \Lip f|_{\Delta^{(0)}}
\]
for every simplex \(\Delta\subset \Sigma\).
\end{definition}

\begin{definition}\label{def:whitney-covering}
Let \(A\subset X\) be a closed subset of a metric space \(X\) and \(\mathcal{B}=(B_i)_{i\in I}\) a covering of \(X\setminus A\). We say that \(\mathcal{B}\) satisfies \(\Whitney(n, \alpha, \delta, \gamma)\) for an integer \(n\geq 0\), and real numbers \(\alpha, \gamma \geq 1\) and \(0 < \delta \leq  1/2\), if the following holds:
\begin{enumerate}[itemsep=0.75em]
\item  every \(B_i\) satisfies 
\[
\diam{B_i} \leq \alpha\cdot r_i,
\]
where \(r_i=d(B_i, A)\);
\item  for every \(x\in X \setminus A\),
\[
x\in N_{\delta \cdot r_i}(B_i)
\]
for at most \(n+1\) indices \(i\in I\);
\item  every \(B_i\) satisfies
\[
\hd(B_i, A)\leq \gamma\cdot r_i,
\]
where \(\hd(B_i, A)\) denotes the asymmetric Hausdorff distance from \(B_i\) to \(A\).
\end{enumerate}   
\end{definition}

Using Definitions~\ref{def:simplicial-target-space} and \ref{def:whitney-covering} we are able to prove the following general Lipschitz extension theorem in a streamlined way.

\begin{theorem}\label{thm:basso-theorem}
Let \(A\subset X\) be a closed subset of a metric space \(X\) and suppose \(f\colon A \to Y\) is a \(1\)-Lipschitz map to a metric space \(Y\). If \(X\setminus A\) admits a covering satisfying \(\Whitney(n, \alpha, \delta, \gamma)\) and \(Y\) is an \((n, C)\)-simplicial extensor, then \(f\) admits a
\begin{equation}\label{eq:desired-Lipschitz-constant}
100\cdot C\cdot \alpha \cdot \delta^{\,-1} \cdot \gamma \cdot \log_2(n+2)
\end{equation}
Lipschitz extension \(F\colon X\to Y\).
\end{theorem}

\begin{proof}
In the following, we use essentially the same proof strategy as in the proof of Theorem~\ref{thm:Lang-Schlichenmaier-explicit} to construct a Lipschitz extension \(F\colon X \to Y\) of \(f\). The restriction of \(F\) to \(
X \setminus A\) will factor through a simplicial complex \(\Sigma\) by definition. To construct this complex, let \((B_i)_{i\in I}\) be a covering of \(X\setminus A\) satisfying \(\Whitney(n, \alpha, \delta, \gamma)\). We consider the open sets \(U_i=N_{\delta r_i}(B_i)\) and \(\Sigma\) will be the nerve of the covering.

Using exactly the same reasoning as in the proof of Lemma~\ref{lem:partition-of-unity}, we find a Lipschitz partition of unity \((\phi_i)_{i\in I}\) subordinated to \((U_i)_{i\in I}\) such that the following holds.  For every \(x\in X \setminus A\),
\[
\sum_{i\in I} \Lip \phi_i (x) \leq 6 \cdot\frac{\log_2(n+2)}{\delta r_j},
\]
where \(j\in I\) is any index such that \(x\in B_{j}\). Let us consider the map 
\[
\Phi \colon X\setminus A \to \Sigma(I)
\]
defined by \(\Phi(x)=(\phi_i(x))_{i\in I}\). Notice that the smallest subcomplex \(\Sigma \subset \Sigma(I)\) containing \(\Phi(X \setminus A)\)  has dimension at most \(n\). 

Let \(\epsilon_0>0\) and select for any \(i\in I\) a point \(a_i\in A\) such that \(d(a_i, B_i) \leq (1+\epsilon_0) r_i\). Let \(\psi\colon \Sigma^{(0)}\to Y \) be defined by \(\psi(e_i)=f(a_i)\). Since \(Y\) is an \((n, C)\)-simplicial extensor, there exists an extension \(\Psi\colon \Sigma \to Y\) of \(\psi\) such that
\[
\Lip \Psi|_{\Delta} \leq C \cdot \Lip \psi|_{\Delta^{(0)}}
\]
for every simplex \(\Delta\subset \Sigma\). Now, let \(F\colon X \to Y\) be defined by \(F(x)=\Psi(\Phi(x))\) if \(x\in X\setminus A\) and \(F(x)=f(x)\) otherwise. We claim that \(F\) satisfies
\begin{equation}\label{eq:pointwise-for-F-everywhere}
\Lip F(x) \leq L,
\end{equation}
for all \(x\in X\), where \(L\) is the constant in \eqref{eq:desired-Lipschitz-constant}.  Due to Lemma~\ref{lem:folklore}, this directly implies that \(F\) is \(L\)-Lipschitz as we may assume that \(X\) is a Banach space. We remark that, due to Lemma~\ref{lem:great-simplification}, it would actually suffice to establish that \(F\) is continuous at every point of \(A\), and to check \eqref{eq:pointwise-for-F-everywhere} only for \(x\in X \setminus A\). However, checking the continuity at \(A\) turns out to be as difficult as establishing the stronger property \(\Lip F(a)\leq L\).

Let \(x\in X \setminus A\). Since \(\phi_i(x)=0\) for all but at most \(n+1\) functions \(\phi_i\), and each \(\phi_i\) is continuous, it follows that there exists an open neighborhood \(U\subset X \setminus A\) of \(x\) such that for any \(x'\in U\) we have that \(\Phi(x)\in \Delta(x')\), where \(\Delta(x')\subset \Sigma\) is the smallest simplex containing \(\Phi(x')\).
Hence, letting \(V\) denote the vertices of \(\Delta(x')\), we have
\[
d(F(x), F(x')) \leq \big[C \cdot \Lip \psi|_{V} \big]\cdot \abs{\Phi(x)- \Phi(x')}. 
\]
Notice that \(e_i\in \Sigma(I)\) belongs to \(V\) if and only if \(i\) is an element of \(I(x')=\{ i\in I : \phi_i(x')>0\}\).  For any \(i\), \(j\in I(x')\) we estimate
\begin{align*}
d(\psi(e_i), \psi(e_j))&=d(f(a_i), f(a_j)) \\
&\leq d(a_i, a_j) \leq 2\big[1+\epsilon_0+\alpha+\delta]\cdot \max_{i\in I(x')} r_i.
\end{align*}
Let \(j\in I\) be such that \(x\in B_j\). Then \(j\in I(x')\), and since
\[
(1-\delta)\cdot r_i \leq d(x', A) \leq d(x, x')+d(x, A)
\]
as well as \(d(x, A)\leq \hd(B_j, A) \leq \gamma\cdot r_j\), we arrive at
\[
\max_{i\in I(x')} r_i \leq \frac{\gamma}{1-\delta} r_j+\frac{1}{1-\delta} d(x, x').
\]
Consequently, by putting everything together and using that \(\abs{\, \cdot \,}\leq \abs{\, \cdot\,}_1\), we find that
\begin{equation*}
\Lip F(x)\leq C \cdot \frac{12 \gamma}{\delta (1-\delta)} [1+\epsilon_0+\alpha+\delta] \cdot \log_2(n+2)
\end{equation*}
In particular,
\begin{equation}\label{eq:final-estimate}
\Lip F(x) \leq 72 \cdot C \cdot \alpha \cdot \delta^{\,-1} \cdot \gamma \cdot \log_2(n+2)+C_0\epsilon_0,
\end{equation}
for every \(x\in X\setminus A\), where we used that \(\delta \in (0, 1/2]\). On the other hand, if \(a\in A\) and \(x\in X\setminus A\), then
\[
d(F(x), F(a))\leq d(F(x), \Psi(e_j))+d(a,x)+ d(x, a_j),
\]
where \(j\in I\) is such that \(x\in B_j\). Since
\[
d(F(x), \Psi(e_j))\leq C  \cdot \big[\max_{i, j\in I(x)} \frac{d(f(a_i), f(a_j))}{\sqrt{2}} \big]\cdot \abs{\Psi(x)-e_j},
\]
we find as above that
\[
d(F(x), \Psi(e_j)) \leq C \cdot \frac{2\gamma}{1-\delta}\big[ 1+\epsilon_0+\alpha+\delta\big] r_j.
\]
Clearly, \(r_j \leq d(x, a)\) and \(d(x, a_j)\leq [1+\epsilon_0+\alpha]\cdot r_j\). Therefore, \eqref{eq:final-estimate} also holds if \(x\in A\). Since \(\epsilon_0>0\) was arbitrary this completes the proof.
\end{proof}

\section{Lipschitz $n$-connectedness}\label{sec:lipschitz-n-connectedness}
In practice, it is often necessary to extend Lipschitz maps which are not defined on \(S^{n-1}\subset \R^n\) but on the boundary \(\partial K\) of a convex body \(K\subset \R^n\). For example, in view of Definition~\ref{def:simplicial-target-space} it seems useful to know how to extend maps from the simplicial boundary of a simplex to the whole simplex. In the following, we investigate property \(\LC(n, \lambda)\) not only for \(B^{n+1}\) but also for other convex bodies as domain spaces.

\begin{definition}Let \(K\subset \R^n\) be a convex body, that is, \(K\) is a compact convex set with non-empty interior. A metric space \(Y\) is said to satisfy condition \(\LC(K, \lambda)\) if every \(L\)-Lipschitz map \(f\colon \partial K \to Y\) admits a \(\lambda L\)-Lipschitz extension \(F\colon K \to Y\).
\end{definition}

It is easy to check that \(\LC(K, \lambda)\) is translation and scale invariant, that is, \(\LC(K, \lambda)\) implies \(\LC(v+sK, \lambda)\) for any \(v\in \R^n\) and \(s>0\). Moreover, it also clearly holds that \(Y\) satisfies \(\LC(n, \lambda)\) if and only if it satisfies \(\LC(B^m, \lambda)\) for every \(m=1, \ldots, n+1\). Since each convex body \(K\subset \R^n\) is bi-Lipschitz equivalent to \(B^n\) via a boundary preserving map, it follows that \(\LC(K, \lambda)\) and \(\LC(B^n, \lambda)\) differ only up to a multiplicative constant. The following lemma makes this more precise. 

\begin{lemma}\label{lem:conversion-ball-to-convex-body}
Let \(K\subset \R^n\) be convex body containing the origin in its interior and \(Y\) a metric space. Let \(\lambda\), \(\lambda_K\) be the smallest constants such that \(Y\) satisfies \(\LC(B^n, \lambda)\) and \(\LC(K, \lambda_K)\), respectively. Then 
\[
\frac{r}{R}\, \lambda \leq \lambda_K \leq \Big(\frac{R}{r}\Big)^2 \lambda,
\]
where \(r=\min_{x\in \partial K} \abs{x}\) and \(R=\max_{x\in \partial K} \abs{x}\) denote the inradius and circumradius of \(K\), respectively. 
\end{lemma}

\begin{proof}
To begin, we show the upper bound. In what follows, we will work with scalings of the Euclidean unit ball \(B^n\subset \R^n\). Let \(t>0\) and abbreviate 
\[
B(t)=\{ x\in \R^n : \abs{x} \leq t\},
\]
as well as \(S(t)=\partial B(t)\) and \(U(t)=B(t)\setminus S(t)\). It follows directly from a theorem of Vrecica, see \cite[Theorem 1]{vrecica--1981}, that the radial projection \(S(t) \to \partial K\) is \(\tfrac{1}{t} \cdot \tfrac{R^2}{r}\)-Lipschitz. In particular, this implies that the radial projection \(\rho\colon K\setminus \{0\} \to \partial K\) is \((\tfrac{R}{r})^2\)-Lipschitz on \(K \setminus U(r)\). 

Let \(f\colon \partial K \to Y\) be an \(L\)-Lipschitz map. We define \(F\colon K\to Y\) by gluing together two maps \(F_1\) and \(F_2\). Let \(F_1\colon K \setminus U(r) \to Y\) be the restriction of \(f\circ \rho\) to \(K \setminus U(r)\). By the above, we find that this defines an \((\tfrac{R}{r})^2 L\)-Lipschitz mapping. On the other hand, since \(Y\) satisfies \(\LC(B(r), \lambda)\), it follows that \(f\circ \rho\) restricted to \(S(r)\) admits a \(\lambda (\tfrac{R}{r})^2 L\)-Lipschitz extension \(F_2\colon B(r) \to Y\). Since \(F_1\) and \(F_2\) agree on the sphere \(S(r)\), it is easy to check that  \(F\colon K \to Y\), which is obtained by pasting together \(F_1\) and \(F_2\), is \((\tfrac{R}{r})^2\lambda L\)-Lipschitz. 

To finish the proof, we show the lower bound using the same methods. Let \(f\colon S(1)\to Y\) be an \(L\)-Lipschitz map. Notice that \(\tfrac{1}{R} K\subset B(1)\). Clearly, the radial projection \(B(1)\setminus \{0\}\to S(1)\) is \(\tfrac{R}{r}\)-Lipschitz on \(B(1)\setminus U(\tfrac{r}{R})\) and thus in particular on \(B(1)\setminus (\tfrac{1}{R} K)^\circ\). Thus, since \(Y\) satisfies \(\LC( \tfrac{1}{R} K, \lambda_K)\), with exactly the same reasoning as above, we find that \(f\) admits an \(\tfrac{R}{r} \lambda_K L\)-Lipschitz extension.
\end{proof}

For each \(n\geq 0\) we let
\begin{equation}\label{eq:simplex-def-2}
\Delta^n=\Big\{ x\in \R^{n+1} : x_i \geq 0 \text{ and } \sum_{i=1}^{n+1} x_i=1 \Big\}    
\end{equation}
denote the standard Euclidean \(n\)-simplex. We will always consider \(\Delta^n\) with the induced Euclidean metric \(d(x, y)=\abs{x-y}\). The following result is a direct consequence of the above lemma, as \(\Delta^n\) is isometric to a regular \(n\)-simplex in \(\R^n\) (whose side length is \(\sqrt{2}\)). 

\begin{lemma}\label{lem:special-case}
Let \(Y\) be a metric space that satisfies \(\LC(n-1, \lambda)\) for some \(n\geq 1\). Then every \(L\)-Lipschitz map \(f\colon \partial \Delta^n \to Y\) admits an \(n^2 \lambda L\)-Lipschitz extension \(F\colon \Delta^n \to Y\).
\end{lemma}

\begin{proof}
We write \(H\) for the affine hyperplane \(\{x\in \R^{n+1} : \sum_{i=1}^{n+1} x_i =1\}\) and consider it with the metric induced by \(\R^{n+1}\). Fix an isometry \(\phi\colon H \to \R^n\). Then \(K=\phi(\Delta^n)\) is a convex body with inradius 
\[
r=\frac{1}{\sqrt{n(n+1)}}
\]
and circumradius \(R=\sqrt{n/(n+1)}\). Hence, the existence of the desired Lipschitz extension \(F\colon \Delta^n \to Y\) follows directly from Lemma~\ref{lem:conversion-ball-to-convex-body}. 
\end{proof}

\subsection{Explicit constants for \(\NPC\) spaces } Prime examples of metric spaces that are Lipschitz \(n\)-connected for every \(n\geq 0\) are complete metric spaces of generalized non-positive curvature. A straightforward argument due to Schlichenmaier \cite[Proposition 6.2.2]{schlichenmaier--2005} shows that such spaces satisfy \(\LC(n, 3)\) for any \(n\). In other words, they are Lipschitz \(n\)-connected for every \(n\) with a constant that does not depend on \(n\). 

Schlichenmaier's argument makes heavy use of the fact that complete \(\NPC\) spaces admit a conical bicombing, a suitable collection of geodesics in a metric space. We proceed by recalling the definition of this notion, see \cite{lang-descombes--2015, basso2020extending} for more information. A map 
\[
\sigma\colon X\times X \times [0,1] \to X
\]
is called a bicombing, if for all \(x\), \(y\in X\), the path \(\sigma_{xy}(\cdot)=\sigma(x, y, \cdot)\) is a metric geodesic connecting \(x\) to \(y\), that is, \(\sigma_{xy}(0)=x\), \(\sigma_{xy}(1)=y\) and \(d(\sigma_{xy}(s), \sigma_{xy}(t))=\abs{s-t}\cdot d(x, y)\) for all \(x\), \(y\in X\). Essentially, a bicombing is a collection of distinguished geodesic in metric space. A useful example to keep in mind is the bicombing on a Banach given by oriented linear segments. Following Descombes and Lang \cite{lang-descombes--2015}, we say that \(\sigma\) is a conical bicombing if
\[
d(\sigma_{xy}(t), \sigma_{x' y'}(t))\leq (1-t)\cdot d(x, x')+t\cdot d(y, y')
\]
for all \(x\), \(y\), \(x'\), \(y'\in X\) and all \(t\in [0,1]\). By a well-known construction going back to Menger \cite[Section 6]{menger--1928}, it follows directly from the definition that complete metric spaces of generalized non-positive curvature admit conical bicombings. See \cite[Section 5]{basso2024noncompactconvexhullgeneralized} for the details. By a refined analysis of Schlichenmaier's arguments, we show in the following that the constant \(3\) from above can be improved to \(\sqrt{3}\). For each non-negative integer \(n\) we define
\[
c_{n}=\frac{1}{\Vol(S^n)}\int_{S^n} \abs{x-q}\, \mathrm{d}x, 
\]
where \(q\) denotes an arbitrary point of \(S^n\). For example, \(c_1=4/\pi\) and \(c_2=4/3\).

\begin{proposition}\label{prop:one}Every complete \(\NPC\) space satisfies \(\LC(B^{n+1}, \lambda_{n})\) for 
\[
\lambda_{n}=\sqrt{1+c_{n}^2}.
\]
In particular, it follows that such spaces are Lipschitz \(n\)-connected with constant \(\sqrt{3}\) for every \(n\).
\end{proposition}

The Lipschitz extensions of \(f\colon S^n \to Y\) that we construct in the proof of Proposition~\ref{prop:one} are of the following type. We say that 
\(F\colon B^{n+1}\to Y\) is a conical extension of \(f\) if there exist a point \(p\in Y\) and a conical bicombing \(\sigma\) on \(Y\), such that \(F(0)=p\) and
\begin{equation}\label{eq:conical-extension}
F(x)=\sigma\big(p, f\bigr(\tfrac{x}{\abs{x}}\bigl), \abs{x}\big),
\end{equation}
for all \(x\in B^{n+1}\setminus\, 0\). It is easy to check that if \(p\in f(S^n)\) then \(\Lip F \leq 3\cdot \Lip f\). We show in the following lemma that the Lipschitz constant of \(F\) depends only on the maximal distance \(R\) between \(p\) and a point contained in \(f(S^n)\).

\begin{lemma}\label{prop:ext1}
Let \(Y\) be a metric space admitting a conical bicombing \(\sigma\) and suppose that \(f\colon S^n \to Y\) is a \(1\)-Lipschitz map.
Fix \(p\in Y\) and define 
\[
R=\sup_{x\in S^n} d(p, f(x)).
\]
Then \(F\colon B^{n+1}\to Y\) defined as in \eqref{eq:conical-extension} satisfies
\begin{equation}\label{eq:lip1}
\Lip F \leq \sqrt{1+R^2}.
\end{equation} 
Moreover, if \(Y=(\mathcal{P}_1(S^n), W_1)\), for \(n\geq 1\), and \(f\colon S^n \to Y\) is defined by \(f(x)=\delta_x\), then \eqref{eq:lip1} becomes an equality. 
\end{lemma}

Proposition~\ref{prop:one} now easily follows from Lemma~\ref{prop:ext1}.

\begin{proof}[Proof of Proposition~\ref{prop:one}]
Let \(Y\) be a complete metric space of generalized non-positive curvature and \(f\colon S^n \to Y\) a Lipschitz map. Since the class of \(\NPC\) spaces is invariant under scaling, we may suppose that \(f\) is \(1\)-Lipschitz. In the following, we construct \(p\in Y\) such that \(d(p, f(q))\leq c_n\) for every \(q\in S^n\). Let \(\rho\) denote the normalized Riemannian volume measure on \(S^n\) and let \(\beta \colon \mathcal{P}_1(Y) \to Y\) be a barycenter map. Furthermore, let \(\mu=f_\ast \rho\) be the push-forward measure of \(\rho\) under \(f\). We set \(p=\beta(\mu)\). Then for every \(q\in S^n\), we have
\[
d(p, f(q))\leq W_1(\mu, \delta_{f(q)})\leq W_1(\rho, \delta_q)=\int_{S^n} \abs{x-q} \, \rho(dx)=c_{n}.
\]
Therefore, according to Lemma~\ref{prop:ext1}, it follows that \(f\) admits a conical extension \(F\colon B^{n+1}\to Y\)  with the desired Lipschitz constant. To finish the proof it suffices to show that \(c_n\leq \sqrt{2}\) for every \(n\). Clearly,
\[
c_n=\frac{1}{2}\int_{S^n} \abs{x-q}+\abs{x+q}\, \rho(dx) \leq \frac{\sqrt{2}}{2} \int_{S^n} \sqrt{\abs{x+q}^2+\abs{x-q}^2}\, \rho(dx),
\]
and so using the parallelogram law, we find that \(c_n\leq \sqrt{2}\), as desired.
\end{proof}

The idea of considering points \(p\) of the form \(p=\beta(\mu)\) in the above proof is due to Urs Lang.

\begin{proof}[Proof of Lemma~\ref{prop:ext1}]
We may suppose that \(n\geq 1\). Let \(x\), \(y\in B^{n+1}\) be distinct points and abbreviate \(r=\abs{x}\) and \(s=\abs{y}\).  Without loss of generality, \(x\), \(y\in B^2\), \(0\leq r\leq s\) and \(x=(r,0)\), as well as \(y=(s\cos(\phi), s \sin(\phi))\) for some \(\phi\in [0, \pi]\). We set \(z=(r\cos(\phi), r\sin(\phi))\). We estimate
\begin{align*}
d(F(x), F(y))&\leq d(F(x), F(z))+d(F(z), F(y)) \\
&\leq \abs{x}\cdot d(f(\tfrac{x}{\abs{x}}), f(\tfrac{y}{\abs{y}}))+(\abs{y}-\abs{z})\cdot d(p, f(\tfrac{y}{\abs{y}})) \\
&\leq \abs{z-x}+d(p, f(\tfrac{y}{\abs{y}}))\cdot\abs{z-y}. 
\end{align*}
Thus, by virtue of the law of cosines, we get

\begin{equation}\label{eq:one}
\frac{d(F(x), F(y))}{\abs{x-y}}\leq \frac{ 2r \sin(\tfrac{\phi}{2})+R(s-r)}{\sqrt{r^2+s^2-2 r s\cos(\phi) }}.
\end{equation}
Let \(f(R,r,s,\phi)\) denote the right hand side of \eqref{eq:one}. We claim that  
\begin{equation}\label{eq:two}
f(R,r,s,\phi)\leq \sqrt{1+R^2}. 
\end{equation} 
Notice that this directly implies \eqref{eq:lip1}. We abbreviate \(D=1+R^2\). Clearly, \eqref{eq:two} follows if 
\begin{equation}\label{eq:three}
\bigl(2r \sin\bigl(\frac{\phi}{2}\bigr)+R(s-r)\bigr)^2-D \bigl(r^2+s^2-2 r s \cos (\phi)\bigr)\leq 0.
\end{equation}
By invoking the identity \(1-\cos(\phi)=2 \sin ^2(\frac{\phi}{2})\) and setting \(\varepsilon=s-r\), we see that the left hand side of \eqref{eq:three} transforms to
\[
\bigl(4(1-D)r^2-4D\varepsilon r\bigr) \cdot \sin ^2(\tfrac{\phi}{2})+4R \varepsilon r \cdot \sin (\tfrac{\phi}{2})+\varepsilon^2(R^2-D).
\]
The discriminant \(\Delta\) of the polynomial \(P(t)=(4(1-D)r^2-4D\varepsilon r)\cdot t^2+4R \varepsilon r  \cdot t+\varepsilon^2(R^2-D)\) equals 
\[
\Delta=-16 D r \varepsilon^2 \bigl((\varepsilon+r) \left(D-R^2\right)-r\bigr)=-16 (1+R^2) r \varepsilon^3
\]
Thus, \(P(\sin(\tfrac{\phi}{2}))\leq 0\) and 
\eqref{eq:three} follows.

Now, suppose that \(Y=(\mathcal{P}_1(S^n), W_1)\) and \(f(x)=\delta_x\) for all \(x\in S^{n}\). We set \(\mu\coloneqq p\). Let \(x\in B^{n+1}\) and suppose \(y\in S^n\) is a point such that
\(W_1(\mu, \delta_y)=R\). Let \(r\), \(s\) and \(\phi\) be defined as above. Notice that
\begin{align*}
W_1(F(x), F(y))&=W_1\bigl((1-r)\mu+r\delta_x, (1-s)\mu+s\delta_y\bigr)\\
&=(s-r)\cdot W_1(\mu, \delta_y)+r\cdot d(x,y).
\end{align*}
Therefore, 
\[
\frac{W_1(F(x), F(y))}{\abs{x-y}}=f(R,r,s,\phi).
\] 
To finish the proof we need to show that
\begin{equation}\label{eq:sup1}
\sup f(R,r,s,\phi)=\sqrt{1+R^2},
\end{equation}
where the \(\sup\) is taken over all \(0 \leq r \leq s \leq 1\) and \(0 \leq \phi \leq \pi\) such that \(f(R,r,s,\phi)\) is a well-defined real number. 

Let \(P(t)\) and \(\Delta\) be defined as above with \(D\) replaced by \(D^\prime=\eta+R^2\), where \(\eta\in [0,1)\). If \(0<\varepsilon < (\tfrac{1}{\eta}-1)r\), then \(\Delta > 0\); thus, since
\[
\frac{4R \varepsilon r}{-2\bigl(4(1-D^\prime)r^2-4D^\prime\varepsilon r\bigr)}\leq \frac{R}{2D^\prime}\leq 1,
\]
there exists \(t\in [0,1]\) such that \(P(t)>0\). Hence, there exists \((r_0, s_0, \phi_0)\) such that 
\[
f(R, r_0, s_0, \phi_0) > \sqrt{D^\prime}
\]
and thereby \eqref{eq:sup1} follows. 
\end{proof}

The following result shows that the constants \(\lambda_n\) in Proposition~\ref{prop:one} are the best possible when only conical extensions are considered. 

\begin{lemma}
Let \(f\colon S^n \to \mathcal{P}_1(S^n)\) be defined by \(f(x)=\delta_x\). Then every conical extension \(F\) of \(f\) satisfies
\[
\Lip F \geq \sqrt{1+c_n^2}.
\]
\end{lemma}

\begin{proof}
We set \(\mu=F(0)\) and denote by \(\rho\) the normalized Riemannian volume measure on \(S^n\). By Lemma~\ref{prop:ext1}, we know that \(\Lip F=\sqrt{1+R^2}\), where \(R=\sup_{x\in S^n} W_1(f(x), \mu)\). Since
\[
W_1(\delta_x, \mu)\leq R,
\]
we can integrate with respect to \(\rho\) on both sides and obtain
\[
\int_{S^n}\Big[\int_{S^n} \abs{x-y}\, \mu(dy)\Big]\, \rho(dx) \leq R.
\]
Hence, by Fubini's theorem,
\[
c_n=\int_{S^n}\Big[\int_{S^n} \abs{x-y}\, \rho(dx)\Big]\, \mu(dy)\leq R.
\]
This completes the proof.
\end{proof}

The idea of considering the \(1\)-Wasserstein space of \(S^n\) as a target space with (potential) extremal properties was motivated by \cite[Lemma~3.1]{basso--2022}, which states that \(\mathcal{F}(S^n)\) has such behavior among Banach space targets.


\section{An explicit version of Lang and Schlichenmaier's theorem}\label{sec:explicit-lang-schlichenmaier}

In this section we prove Theorem~\ref{thm:LS-msot-general-version} from the introduction. The overall strategy is to prove Theorem~\ref{thm:LS-msot-general-version} using Theorem~\ref{thm:basso-theorem}. 

Let \(X\) be a metric space and \(A\subset X\) a closed subset, and suppose there exist \(n\geq 1\) and \(c\geq 0\), such that \(A\) satisfies \(\Nagata(n-1, c)\). Furthermore, let \(Y\) be a metric space that satisfies \(\LC(n-1, \lambda)\) for some \(\lambda\geq 1\). Note that unlike in the statement of Theorem~\ref{thm:LS-msot-general-version}, we have chosen to work with the parameter \(n-1\) instead of \(n\). This choice allows us to write \(n\) instead of \(n+1\) in many places below. 

In order to apply Theorem~\ref{thm:basso-theorem}, we need to demonstrate that \(Y\) is an \((n, C)\)-simplicial extensor and \(X \setminus A\) admits a covering satisfying \(\Whitney(n, \alpha, \delta, \gamma)\). Notice that because of  Proposition~\ref{prop:Theorem5.2Revisited}, we already know that \(X\setminus A\) admits a covering satisfying \(\Whitney(3n, \alpha, \delta, \gamma)\). Our main effort will be devoted to reducing the multiplicity of the covering from \(3n\) to \(n\), while retaining the properties of a Whitney covering. But before doing so, let us begin with the straightforward proof that \(Y\) is an \((n, C)\)-simplicial extensor.

\begin{proposition}\label{prop:iterated-simplicial-extension}
Let \(\Sigma\) be an \(n\)-dimensional simplicial complex equipped with the \(\ell_2\)-metric and suppose \(Y\) is a metric space satisfying \(\LC(n-1, \lambda)\). Then every map \(f\colon \Sigma^{(0)}\to Y\) admits an extension \(F\colon \Sigma \to Y\) such that  
\begin{equation}\label{eq:proved-by-induction}
\Lip F|_{\Delta} \leq \big[ \lambda^n \cdot (\sqrt{2})^{n-1} \cdot \sqrt{n}\cdot (n\,!)^2 \big] \cdot \Lip f|_{\Delta^{(0)}}
\end{equation}
for each simplex \(\Delta\subset \Sigma\). In other words, up to an uniform constant the Lipschitz constant of the restriction \(F|_{\Delta}\) depends only on the Lipschitz constant of \(f\) restricted to the vertices of \(\Delta\).
\end{proposition}


The main reason why such a result is possible is that the combinatorial structure of \(\Delta^n\) lends itself to inductive extension of maps. The proof of Proposition~\ref{prop:iterated-simplicial-extension} relies on iterated applications of the following auxiliary lemma.

\begin{lemma}\label{lem:auxiliary}
Let \(Y\) be a metric space that satisfies \(\LC(n-1, \lambda)\) for some \(n\geq 2\) and \(f\colon \partial \Delta^n \to Y \) a map which is \(L\)-Lipschitz when restricted to any \((n-1)\)-face of \(\Delta^n\). Then \(f\) admits a 
\[
\sqrt{2+2/(n-1)}\cdot n^2\cdot \lambda \cdot L
\]
Lipschitz extension \(F\colon \Delta^n \to Y\). 
\end{lemma}

\begin{proof}
Let \(x\), \(y\in \partial \Delta^n\). By a result of \cite[Theorem 1]{baader2020distortion}, there
exists a continuous path \(\gamma\colon [0,1] \to \partial \Delta^n\) such that 
\[
\ell(\gamma)\leq \sqrt{2+\tfrac{2}{n-1}} \cdot d(x, y),
\]
where \(\ell(\gamma)\) denotes the length of \(\gamma\). We abbreviate \(c=\sqrt{2+2/(n-1)}\). With the help of \cite[Lemma 2.4]{basso--wenger--young}, we obtain
\[
d(f(x), f(y))\leq \ell(f\circ \gamma) \leq \big[\hspace{-0.35em}\sup_{p\in \partial \Delta^n} \Lip f(p)\big] \cdot \ell(\gamma) \leq  L c \cdot d(x, y).
\]
Therefore, \(f\) is \(Lc\)-Lipschitz and the existence of \(F\) follows directly from Lemma~\ref{lem:special-case}.
\end{proof}

\begin{proof}[Proof of Proposition~\ref{prop:iterated-simplicial-extension}]
In case \(\Sigma\) is zero-dimensional there is nothing to prove, so we may assume that \(n>0\). Since \(Y\) satisfies \(\LC(n-1, \lambda)\) and thus \(\LC(0, \lambda)\) in particular, there exists a map \(F\colon \Sigma^{(1)}\to Y\) such that
\begin{equation*}
\Lip F|_{\Delta} \leq \lambda \cdot \Lip f|_{\Delta^{(0)}}
\end{equation*}
for every simplex \(\Delta\subset \Sigma^{(1)}\). This proves the proposition in case \(n=1\). Suppose now \(n\geq 2\). By induction, we may suppose that there exists a map \(\widetilde{F}\colon \Sigma^{(n-1)}\to Y\) such that \eqref{eq:proved-by-induction} holds for every \((n-1)\)-simplex \(\Delta\subset \Sigma\). Hence, by applying Lemma~\ref{lem:auxiliary} to \(\widetilde{F}|_{\partial \Delta}\) for each \(n\)-simplex \(\Delta\subset \Sigma\), we obtain a map \(F\colon \Sigma\to Y\) such that 
\[
\Lip F|_{\Delta} \leq \big[\sqrt{2+2/(n-1)}\cdot n^2\cdot \lambda\big]\cdot \max_{\Delta'\subset \Delta^{(n-1)}} \Lip \widetilde{F}|_{\Delta'}.
\]
This implies that
\[
\Lip F|_{\Delta} \leq \big[ \lambda^n \cdot (\sqrt{2})^{n-1} \cdot \sqrt{n}\cdot (n\,!)^2 \big] \cdot \Lip f|_{\Delta^{(0)}},
\]
as desired. 
\end{proof}

We now proceed with the much more involved proof that \(X\setminus A\) admits a \(\Whitney(n, \alpha, \delta, \gamma)\)-covering. As a first step in this direction, we establish the following equivalent characterization of \(\Nagata(n,c)\). For a list of other equivalent characterizations, see \cite[Proposition 4.2]{licht2023lipschitz}.

\begin{lemma}\label{lem:colored-nagata}
Let \(X\) be a metric space satisfying \(\Nagata(n,c)\). Then for every \(s>0\), \(X\) admits a covering \(\mathcal{B}\) such that the following holds:
\begin{enumerate}
    \item every \(B\in \mathcal{B}\) satisfies 
    \[
    \diam B \leq 2(c+1)(n+2)\cdot s;
    \]
    \item\label{it:separated-sets} there exists a decomposition \(\mathcal{B}=\mathcal{B}_1 \cup \dotsm \cup \mathcal{B}_{n+1}\) such that for every \(k=1, \ldots, n+1\),
    \[
    d(B, B')>s
    \]
    for all distinct \(B\), \(B'\in \mathcal{B}_k\).
\end{enumerate}
\end{lemma}
The numbers \(k=1, \ldots, n+1\) are sometimes referred to as the colors of the covering \(\mathcal{B}\), and we say that all members of \(\mathcal{B}_k\) have the same color. In \cite[Proposition 2.5]{lang-schlichenmaier}, Lang and Schlichenmaier gave a proof of Lemma~\ref{lem:colored-nagata} without explicit constants by relying on geometric properties of simplicial complexes. In the following we give a simplified proof of the lemma by working with cubical complexes instead.

\begin{proof}[Proof of Lemma~\ref{lem:colored-nagata}]
Fix \(s>0\) and let \((B_i)_{i\in I}\) be a \(c\cdot (2s)\)-bounded covering of \(X\) with \(2s\)-multiplicity at most \(n+1\). We may suppose that \(I\neq \varnothing\) in the following. We consider the functions \(\phi_i\colon X \to \R\) defined by 
\[
\phi_i(x)=\max\{ s-d(x, B_i), 0\}.
\]
It follows that \(\phi_i \leq s\), and for each \(x\in X\), we have \(\phi_i(x)>0\) for at most  \(n+1\) indices \(i\in I\). Thus, \(\phi \colon X \to \prod_{i\in I} [0, 2s]\) defined by \(\phi(x)=(\phi_i(x))_{i\in I}\) maps to the \((n+1)\)-skeleton \(K\subset C\) of the (possibly infinite-dimensional) cube \(C=\prod_{i\in I} [0, 2s]\). In what follows, we view \(\phi\) as a map from \(X\) to \(K\) and endow \(K\) with the uniform distance \(\rho( (\alpha_i), (\beta_i))=\sup_i \abs{\alpha_i-\beta_i}\). Clearly, \(\phi\) is \(1\)-Lipschitz. 

Let \(\text{sd}_1(K)\) be the first cubical barycentric subdivision of \(K\). Each \(v\in \text{sd}_1(K)\) is the barycenter of a unique cube \(K_v\) of \(K\). Given \(v=(v_i)\in \text{sd}_1(K)\) we call \(i\in I\) an active coordinate of \(v\) if \(v_i=s\). Notice that number of active coordinates of \(v\) is equal to the dimension of \(K_v\).  In the following, we abbreviate 
\[
s_{k}=\Big(1-\frac{k}{n+2}\Big)\cdot s
\]
for \(k=1, \ldots, n+1\). Notice that \(s_{k+1} < s_k < s\). For \(v\in \text{sd}_1(K)\) with \(k=\dim K_v >0\), we consider the rectangular cell \(R_v\) consisting of all \(y\in K\) such that for every active coordinate
\[
\abs{v_i-y_i}< s-s_k,
\]
and \(\abs{v_i-y_i} \leq s_{k+1}\) otherwise (with the convention that \(s_{n+2}=0\)). Let \(B_v\subset X\) denote the set of all \(x\in X\) such that \(\phi(x)\in R_v\).

For each \(k=1, \ldots, n+1\) we define \(\mathcal{B}_k\) to be the family consisting of all sets \(B_v\) for which \(K_v\) has dimension \(k\). By construction,
\[
d(B_v, B_{v'})\geq \rho(R_v, R_{v'})
\]
and so if \(v\neq v'\) and \(\dim K_v=\dim K_{v'}=k\), it follows that
\[
d(B_v, B_{v'}) \geq \rho(R_v, R_{v'}) > \abs{s_{k}-s_{k+1}}=\frac{1}{n+2} \cdot s.
\]
This implies Lemma~\ref{lem:colored-nagata}\eqref{it:separated-sets}. Moreover if \(x\), \(x'\in B_v\), then for any active coordinate \(i\) of \(v\), we have \(\phi_i(x), \phi_i(x')>0\). Hence, since there is at least one active coordinate, using the definition of \(\phi_i\), we find that \(\diam B_v \leq (c+1) 2s\). 

To finish the proof we need to establish that \(\mathcal{B}=\bigcup \mathcal{B}_k\) covers \(X\). Let \(y=\phi(x)\) for some \(x\in X\). We consider an auxiliary set \(A(y)\subset \N\) defined as follows. We let \(A(y)\) be the set
\[
\bigl\{ k\in \N: \exists v\in \text{sd}_1(K) \text{  s. t. } \# I(v)=k \,\text{ and }\, \forall i\in I(v): \abs{v_i-y_i} < s-s_k \, \bigr\},
\]
where we use \(I(v)\) to denote the set of all active coordinates of \(v\). Then \(1\in A(y)\) and \(A(y)\) is contained in the interval \([1, n+1]\). Let \(k\) denote the maximal element of \(A(y)\) and let \(v\in \text{sd}_1(K)\) be such that \(\#I(v)=k\) and \(\abs{v_i-y_i} < s-s_k\) for all \(i\in I(v)\). We may suppose that \(v_i=0\) for all \(i\in I\setminus I(v)\). We claim that \(\phi(x)\in R_v\) (and so \(x\in B_v\)). Suppose for the sake of a contradiction that this is not the case. Then \(k<n+1\) and there exists \(j\in I\setminus I(v)\) such that \(\phi_j(x) > s_{k+1}\). However, using that \(\phi_j(x) \leq s\), this implies that \(\abs{s-\phi_{j}(x)} < s-s_{k+1}\). By considering \(v'\in \text{sd}_1(K)\), which is obtained from \(v\) by replacing its \(j\)-th entry with \(s\), we find that \(k+1\in A(y)\). This is a contradiction. Therefore, \(\phi(x)\in R_v\), as desired.
\end{proof}

Our next result is the following refined version of Proposition~\ref{prop:Theorem5.2Revisited}. 

\begin{proposition}\label{prop:refined-whitney-decomposition}
Let \(X\) be a metric space and \(A\subset X\) a closed subset satisfying \(\Nagata(n-1,c)\) for some \(n\geq 1\) and \(c\geq 0\). 
Then for every 
\[
r> 2\cdot (c+1) \cdot 4^{n+1}
\]
there exists a covering \(\mathcal{B}\) of \(X\setminus A\) such that
\vspace{0.2em}
\begin{enumerate}[itemsep=0.75em]
\item\label{it:one-a} every \(B\in \mathcal{B}\) satisfies \(\diam{B} \leq \alpha\cdot d(B, A)\), where
\[\alpha= 40\cdot r^3\cdot (c+1) \cdot (n+1);
\]
\item\label{it:two-a} every \(E\subset X \setminus A\) with 
\[
\diam E \leq r^{\,\frac{1}{2n}}\cdot d(E, A)
\]
meets at most \(n+1\) members of \(\mathcal{B}\);
\item\label{it:three-a} for every \(B\in \mathcal{B}\), we
\[
\hd(B, A)\leq r^2\cdot d(B, A),
\]
where \(\hd(B, A)\) denotes the asymmetric Hausdorff distance from \(B\) to \(A\).
\end{enumerate}   

\end{proposition}

Notice in particular the improved bound for the multiplicity in \eqref{it:two-a}. This improvement is achieved by modifying the proof of Proposition~\ref{prop:Theorem5.2Revisited} by using the coverings from Lemma~\ref{lem:colored-nagata} above. The following proof of Proposition~\ref{prop:refined-whitney-decomposition} is essentially the same as the proof of \cite[Theorem 1.6]{lang-schlichenmaier} by Lang and Schlichenmaier. However, our presentation is quite different and hopefully more transparent.

\begin{proof}[Proof of Proposition~\ref{prop:refined-whitney-decomposition}]
 Let us begin by fixing some notation. For \(k\in \{0, 1, \ldots, n-1\}\) and \(i\in \Z\) we define the “annulus"
\[
R_k^i=\{ x\in X \, : \, r^{i-1}\cdot r^{k/n} \leq d(x, A) < r^{i}\cdot r^{k/n} \}.
\]
We point out that for every \(k\) the union \(\bigcup_{i\in \Z} R_k^i\) is a partition of \(X\setminus A\). Also notice that \(R_k^i\subset R_0^{i}\cup R_0^{i+1}\) for every \(i\in \Z\). Now, suppose for the moment that we have constructed families \(\mathcal{B}_k=\bigcup_{i\in \Z} \mathcal{B}_k^i\) with the following properties: 
 \begin{enumerate}[label=\alph*)]
     \item\label{it:rk} every \(B\in \mathcal{B}_k^i\) is contained in \(R_k^i\);
     \item\label{it:mult} every family \(\mathcal{B}_k^i\) has \(r^{i+1}\)-multiplicity one;
     \item\label{it:cover} the union \(\mathcal{B}=\mathcal{B}_0 \cup \dotsm \cup \mathcal{B}_{n-1}\) is a covering of \(X\setminus A\).
 \end{enumerate}
 We claim that any such covering \(\mathcal{B}\) satisfies \eqref{it:two-a}, that is, if \(E\subset X\setminus A\) has diameter less than or equal to \(r^{1/(2n)}\cdot  d(E, A)\), then \(E\) meets at most \(n+1\) members of the cover. This can be seen as follows. First notice that since \(r > \varphi^{2n}\), where \(\varphi\) denotes the golden ratio, we have \(1+r^{1/(2n)} < r^{1/n}\). Hence, there exists an integer \(\ell\) such that
 \[
 r^{\ell/n} \leq d(e, A) <  r^{\ell/n}\cdot r^{2/n}
 \]
 for every \(e\in E\). Let \(k_0 \in \{0, \ldots, n-1\}\) be the unique integer such that \(k_0 \equiv \ell+1 \pmod{n}\). It follows that for every \(k\) distinct from \(k_0\) there exists a unique \(i\in \Z\) such that \(E\subset R_{k}^i\). Thus, since \(\mathcal{B}_{k}^i\) has \(r^{i+1}\)-multiplicity one, we find that \(E\) meets at most one member of \(\mathcal{B}_{k}\) for \(k \neq k_0\). Notice that \(E\subset R_{k_0}^i \cup R_{k_0}^{i+1}\) for some \(i\in \Z\). In particular, using that \(\diam E \leq r^{i+1}\), we find that \(E\) meets at most two members of \(\mathcal{B}_{k_0}\). Consequently, \(E\) meets at most \(n+1\) members of \(\mathcal{B}\), as claimed. 

 In what follows, we construct families \(\mathcal{B}_k\) satisfying the properties stated above. Fix a retraction \(\rho\colon X \to A\) such that \(d(x, \rho(x)) \leq r^{i}\) on \(R_0^i\).  Let \(c'=2(c+1)(n+1)\) and for every \(i\in\Z\) we set 
 \[
 s_i=4 \cdot r^{i+1}
 \]
 and we let \(\mathcal{A}^i=\mathcal{A}^i_0 \cup \dotsm \cup \mathcal{A}_{n-1}^i\) be a \(c'\cdot s_i\)-bounded covering of \(A\) such that each \(\mathcal{A}_k^i\) has \(s_i\)-multiplicity one. The existence of such colored coverings of \(A\) is guaranteed by Lemma~\ref{lem:colored-nagata}. We define 
 \[
 \widehat{\mathcal{B}}_k^{\,i}=\{ \, \rho^{-1}(C)\cap  R^i_k \,  : \,  C\in \mathcal{A}_k^i \cup \mathcal{A}_{k}^{i+1}\,\}.
 \]
Let \(\sim\) be the equivalence relation on \(\widehat{\mathcal{B}}_k^{\,i}\) generated by \(B \sim_\ast B'\) if and only if \(d(B, B') \leq r^{i+1}\). Using this relation, we define
\[
\mathcal{B}_k^i=\big\{ \bigcup_{B'\sim B} B' \, : \, B\in  \widehat{\mathcal{B}}_k^{\,i}\,\big\}.
\]
We need to justify why \(\mathcal{B}=\bigcup \mathcal{B}_k^i\) has the desired properties. Obviously, \ref{it:rk} and \ref{it:mult} are satisfied because of the construction of \(\mathcal{B}_k^i\).  Let \(x\in X\setminus A\). Then \(x\) is contained in some \(R_0^i\) and there exists \(C\in \mathcal{A}^i\) containing \(\rho(x)\). Let \(k\) denote the color of \(C\). It follows that there exists \(B\in \widehat{\mathcal{B}}_k^{\,i-1} \cup \widehat{\mathcal{B}}_k^{\,i}\) such that \(x\in B\). In particular, \(x\) is contained in some member of \(\mathcal{B}_k^{i-1}\cup \mathcal{B}_k^i\) and thus in some member of \(\mathcal{B}\). This proves that \(\mathcal{B}\) is a covering of \(X\setminus A\).

To finish the proof we need to show that \(\mathcal{B}\) also satisfies \eqref{it:one-a} and \eqref{it:three-a}. Given \(B\in \widehat{\mathcal{B}}_k^{\,i}\), it holds \(B=\rho^{-1}(C)\cap R_k^i\)  for some \(C\in \mathcal{A}^j_k\) with \(j\in \{i, i+1\}\).  We call the smallest such \(j\) the order of \(B\). We claim that any equivalence class \([B]\) contains at most one element of order \(i+1\). Notice that if \(B\in  \widehat{\mathcal{B}}_k^{\,i}\) has order \(j\), then
\[
\diam B \leq 2 r^{i+1}+c' s_j.
\]
Moreover, if \(B'\in \widehat{\mathcal{B}}_k^{\,i}\) is distinct from \(B\) and has the same order, then we compute
\[
s_j <d(C, C')\leq 2 r^{i+1}+d(B, B').
\]
This implies that \(B\) and \(B'\) are \(r^{j+1}\)-separated. Hence, \(B \sim_\ast B'\) implies that \(B\) and \(B'\) have  different orders. As a result, if \(B \sim_\ast B' \sim_\ast B''\), then \(B\) and \(B''\) have the same order \(j\). Thus, letting \(j'\) be the other element  of \(\{i, i+1\}\) distinct from \(j\), we find that
\[
r^{j+1} < d(B, B'') \leq 2r^{i+1}+\diam B' \leq 4 r^{i+1}+4 c' r^{j'+1}.
\]
But since \(r> 4(1+c')\) this is only possible if \(j=i\) and so \([B]\) contains at most one element of order \(i+1\). As a result, for any \(B\in \mathcal{B}_k^i\) we have
\[
\diam B \leq 2\cdot( 2 r^{i+1}+c's_i)+2 r^{i+1}+(2 r^{i+1}+c's_{i+1}).
\]
Since \(r^{i-1} \leq d(B, A)\), we arrive at \(\diam B \leq \alpha \cdot d(B, A)\). This shows \eqref{it:one-a}. By construction,  any \(B\in \mathcal{B}_k^i\) is contained in \(R_k^i\), and thus for every \(b\in B\), 
\[
d(b, A) \leq r^{i}\cdot r^{k/n} \leq r^2\cdot r^{i-1} \leq r^2 \cdot d(B, A).
\]
This establishes \eqref{it:three-a} and finishes the proof. 
\end{proof}

Let \(\mathcal{B}=(B_i)_{i\in I}\) be a covering as in Proposition~\ref{prop:refined-whitney-decomposition}. As in the proof of Theorem~\ref{thm:Lang-Schlichenmaier-explicit}, we will work in the following with the open sets \(U_i=N_{\delta r_i}(B_i)\), where \(\delta>0\) is a sufficiently small constant. Recall that we use the notation \(r_i=d(B_i, A)\). Fix \(x\in X \setminus A\) and let \(I(x)\subset I\) denote those indices \(i\) such that \(x\in U_i\). For any \(i\in I(x)\) we select \(x_i\in B_i\) such that \(d(x, x_i)< \delta r_i\). Since \(r_i \leq d(x_i, A)\) it follows from Proposition~\ref{prop:refined-whitney-decomposition}\eqref{it:three-a} that
\[
r_i \leq  \delta r_i+\delta r_j+r^2 r_j
\]
for all \(i\), \(j\in I(x)\). In particular, letting \(E=\{ x_i : i\in I(x)\}\), we find that
\[
\diam E\leq 2\delta \cdot \sup_{i\in I(x)} r_i \leq \frac{2\delta}{1-\delta} (\delta+r^2) \inf_{i\in I(x)} r_i.
\]
But \(\inf \{ r_i : i\in I(x)\}\) is less than or equal to \(d(E, A)\), and so if \(\delta=\frac{1}{8 r^2}\), then 
\[
\frac{2\delta}{1-\delta} (\delta+r^2)\leq 4\delta \cdot (\delta+r^2)\leq  r^{\frac{1}{2n}}
\]
and Proposition~\ref{prop:refined-whitney-decomposition}\eqref{it:two-a} tells us that \(E\) meets at most \(n+1\) members of \(\mathcal{B}\). In particular, the family \((U_i)_{i\in I}\) has multiplicity at most \(n+1\). To summarize, we have shown that \((B_i)_{i\in I}\) is a covering of \(X\setminus A\) that satisfies \(\Whitney(n, \alpha, \delta, \gamma)\) for
\begin{equation}\label{eq:explicit-values-of-constants}
 \alpha=40\cdot r^3 \cdot (c+1)\cdot (n+1), \quad \delta=(8 r^2)^{-1}, \quad \gamma=r^2.
\end{equation}
Now, we have everything at hand to prove an explicit version of Lang and Schlichenmaier's extension theorem.

\begin{proof}[Proof of Theorem~\ref{thm:LS-msot-general-version}]
Let \(X\), \(Y\) denote metric spaces and suppose \(A\subset X\) is a closed subset. Assume that \(A\) satisfies \(\Nagata(n-1, c)\) and \(Y\) satisfies \(\LC(n-1, \lambda)\). Let \(f\colon A \to Y\) be a \(1\)-Lipschitz map. In the following, we show that \(f\) admits a Lipschitz extension \(F\colon X \to Y\). We set
\[
r=16\cdot (c+1)\cdot 4^n.
\]
By Proposition~\ref{prop:refined-whitney-decomposition} and the comments following it, we find that \(X\setminus A\) admits a covering that satisfies \(\Whitney(n, \alpha, \delta, \gamma)\) for
\[
 \alpha=40\cdot (c+1)^4 \cdot 128^n , \quad \delta^{-1}=2048\cdot(c+1)^2 \cdot 16^n, \quad \gamma=256\cdot (c+1)^2 \cdot 16^n.
\]
Here, we used \eqref{eq:explicit-values-of-constants} and that \(n+1 \leq 2^n\). Proposition~\ref{prop:iterated-simplicial-extension} tells us that \(Y\) is an \((n, C)\)-simplicial extensor for
\[
C= e^2\cdot \lambda^n \cdot 2^{n} \cdot n^{2n+1},
\]
where we used Stirling's approximation for \(n !\) to obtain this particular value for \(C\). Hence, by Theorem~\ref{thm:basso-theorem} it follows that \(f\) admits a Lipschitz extension \(F\colon X \to Y\) such that
\[
\Lip F \leq 3 \cdot 10^{10}\cdot (c+1)^8\cdot (10^5 \cdot \lambda)^n \cdot n^{2(n+1)}.
\]
Now, since \( 10^{5n}\cdot n^{-n} \leq e^{10^5}\), the desired constant in \eqref{eq:estimate-we-worked-for-really-hard} follows. 
\end{proof}


\subsection{Acknowledgements} I am grateful to Urs Lang for useful comments which simplified the proof of Proposition~\ref{prop:one}. I also thank the anonymous referee for his or her careful reading of my paper and their helpful comments.

\subsection{Author contribution}
The author confirms the sole responsibility for writing of this article and its findings.
\subsection{Conflict of interest}
Authors state no conflict of interest.



\let\oldbibliography\thebibliography
\renewcommand{\thebibliography}[1]{\oldbibliography{#1}
\setlength{\itemsep}{2pt}} 

\bibliographystyle{plain}
\bibliography{refs}

\end{document}